\documentclass{article}
\usepackage{paralist}
\usepackage{quoting}
\usepackage{amsthm,amsmath,amssymb}
\usepackage{mathtools}
\allowdisplaybreaks
\usepackage{tikz}
\usepackage{pgfplots}
\usetikzlibrary{matrix,shapes,arrows,positioning,shadows,backgrounds}

\newtheorem{theorem}{Theorem}[section]
\newtheorem*{theorem*}{Theorem}
\newtheorem{intro}{Theorem}
\newenvironment{introtheorem}[1]
  {\renewcommand\theintro{#1}\intro}
  {\endintro}

\newtheorem{proposition}[theorem]{Proposition}
\newtheorem{corollary}[theorem]{Corollary}
\newtheorem{lemma}[theorem]{Lemma}

\theoremstyle{remark}
\newtheorem*{remark}{Remark}

\newcommand{\tO}{\mathtt 0}                            
\newcommand{\tL}{\mathtt 1}                            
\newcommand{\tZ}{\mathtt 2}                            
\newcommand{\tM}{\mathtt 3}                            
\newcommand{\tQ}{\mathtt 4}                            

\newcommand{\ta}{\mathtt a}                            
\newcommand{\tba}{\bar{\mathtt a}}                            
\newcommand{\tb}{\mathtt b}                            
\newcommand{\tbb}{\bar{\mathtt b}}                       
\newcommand{\tc}{\mathtt c}                            
\newcommand{\bt}{\mathbf t}                         
\newcommand{\tpsi}{\sigma}
\newcommand{\diffseq}{\mathbf B} 
\newcommand{\bardiffseq}{\bar{\mathbf B}}
\newcommand{\autoseq}{\mathbf A}
\newcommand{\autoseqplus}{\mathbf A^{\!+}}
\newcommand{\barautoseq}{\overline{\mathbf A}}

\newcommand{\clength}{0.35} 
\newcommand{\Lbleft}{
\begin{tikzpicture}[xscale=0.2,yscale=1,anchor=base, baseline]
\node[inner sep=1pt] (a1)	at (0,0)	{$\tb$};
\draw (a1.south)  -- ++(0,-0.1) --  ++(-\clength,0);
\end{tikzpicture}
}
\newcommand{\Lbright}{
\begin{tikzpicture}[xscale=0.2,yscale=1,anchor=base, baseline]
\node[inner sep=1pt] (a1)	at (0,0)	{$\tb$};
\draw (a1.south)  -- ++(0,-0.1) --  ++(\clength,0);
\end{tikzpicture}
}
\newcommand{\Lbbright}{
\begin{tikzpicture}[xscale=0.2,yscale=1,anchor=base, baseline]
\node[inner sep=1pt] (a1)	at (0,0)	{$\tbb$};
\draw (a1.south)  -- ++(0,-0.1) --  ++(\clength,0);
\end{tikzpicture}
}\newcommand{\Lbbleft}{
\begin{tikzpicture}[xscale=0.2,yscale=1,anchor=base, baseline]
\node[inner sep=1pt] (a1)	at (0,0)	{$\tbb$};
\draw (a1.south)  -- ++(0,-0.1) --  ++(-\clength,0);
\end{tikzpicture}
}

\newcommand{\Lcleft}{
\begin{tikzpicture}[xscale=0.2,yscale=1,anchor=base, baseline]
\node[inner sep=1pt] (a1)	at (0,0)	{$\tc$};
\draw (a1.south)  -- ++(0,-0.1) --  ++(-\clength,0);
\end{tikzpicture}
}
\newcommand{\Lcright}{
\begin{tikzpicture}[xscale=0.2,yscale=1,anchor=base, baseline]
\node[inner sep=1pt] (a1)	at (0,0)	{$\tc$};
\draw (a1.south)  -- ++(0,-0.1) --  ++(\clength,0);
\end{tikzpicture}
}

\DeclareMathOperator{\e}{\mathrm{e}}                   
\DeclareMathOperator{\realpart}{\mathrm {Re}}

\title{Gaps in the Thue--Morse word}
\author{Lukas Spiegelhofer
\\
Montanuniversit\"at Leoben, Austria}

\begin{document}
\maketitle
\begin{abstract}
The Thue--Morse sequence is a prototypical automatic sequence found in diverse areas of mathematics, and in computer science.
We study occurrences of factors $w$ within this sequence, more precisely, the sequence of gaps between consecutive occurrences.
This gap sequence is morphic; we prove that it is not automatic
as soon as the length of $w$ is at least $2$,
thereby answering a question by J.~Shallit in the affirmative.
We give an explicit method to compute the \emph{discrepancy} of the number of occurrences of
the block $\mathtt{01}$ in the Thue--Morse sequence.
We prove that the sequence of discrepancies is the sequence of output sums of a certain base-$2$ transducer.
\end{abstract}
\renewcommand{\thefootnote}{\fnsymbol{footnote}}

\footnotetext{The author was supported by the Austrian Science Fund (FWF), project F5502-N26, which is a part of the Special Research Program `Quasi Monte Carlo methods:  Theory and Applications',
and by the FWF-ANR project ArithRand, grant numbers I4945-N and ANR-20-CE91-0006.}
\renewcommand{\thefootnote}{\arabic{footnote}} 

\section{Introduction and main result}
Automatic sequences can be defined via deterministic finite automata with output (DFAO):
feeding the base-$q$ expansion (where $q\geq 2$ is an integer) of $0,1,2,\ldots$ into such an automaton, we obtain an automatic sequence as its output, and each automatic sequence is obtained in this way.
One of the simplest automatic sequences --- in terms of the size of the defining substitution --- is the Thue--Morse sequence $\bt$.
It is the fixed point of the substitution $\tau$ given by
\begin{equation}\label{eqn_tau}
\tau: \tO\mapsto\tO\tL,\quad \tL\mapsto\tL\tO,
\end{equation}
starting with $\tO$:
\begin{equation}\label{eqn_TM_32}
\mathbf t=\tau^\omega(\tO)=
\tO\tL\tL\tO\tL\tO\tO\tL\tL\tO\tO\tL\tO\tL\tL\tO
\tL\tO\tO\tL\tO\tL\tL\tO\tO\tL\tL\tO\tL\tO\tO\tL
\cdots.
\end{equation}
(Here $\tau^\omega(\tO)$ denotes the point-wise limit of the iterations $\tau^k(\tO)$, in symbols
$\tau^\omega(\tO)\bigr\rvert_j=\lim_{k\rightarrow\infty}\tau^k(\tO)\bigr\rvert_j$.
We use analogous notation in other places too.)
Occurrences of this sequence in different areas of mathematics can be found in the paper~\cite{AS1999} by Allouche and Shallit, which also offers a good bibliography.
Another survey paper on the Thue--Morse sequence was written by Mauduit~\cite{M2001}.
Much more information concerning automatic and general morphic sequences is presented in the book~\cite{AlloucheShallitBook} by Allouche and Shallit.

Terms $1$ through $12$ of the Thue--Morse sequence (the $0$th term is omitted) make a peculiar appearance in \emph{The Simpsons Movie}~\cite{S}:
\begin{quoting}
  Russ Cargill: `\,[\,\ldots] I want $10,000$ tough guys, and I want $10,000$ soft guys to make the tough guys look tougher! And here's how I want them arranged:\\tough, tough, soft, tough, soft, soft, tough, tough, soft, soft, tough, soft.'
\end{quoting}

This could, of course, be a coincidence.
A different, more sensible explanation of this appearance is along the lines of Brams and Taylor~\cite[36--44]{BT1999}.
They rediscover the Thue--Morse sequence while seeking \emph{balanced alternation} between two parties `Ann' and `Ben'. However, Brams and Taylor do not attribute the resulting sequence to Thue and Morse (and neither to Prouhet).

We are interested in counting the number of times that a word occurs as a \emph{factor} --- a contiguous finite subsequence --- of another word; 
a related concept is the \emph{binomial coefficient of two words}, which counts the corresponding number concerning general subsequences.
We wish to note the related paper by Rigo and Salimov~\cite{RigoSalimov2015},
defining $m$-\emph{binomial equivalence} of two words,
and the later paper by Lejeune, Leroy, and Rigo~\cite{LejeuneLeroyRigo2020},
where the Thue--Morse sequence is investigated with regard to this new concept.

It is well known that the sequence $\bt$ is \emph{uniformly recurrent}~\cite[Section 1.5.2]{Lothaire2002}.
That is, for each factor $w$ of $\bt$ there is a length $n$ with the following property: every contiguous subsequence of $\bt$ of length $n$ contains $w$ as a factor.
The factor $\tO\tL$ therefore appears in $\bt$ with bounded distances.
We are interested in the infinite word $\diffseq$ (over a finite alphabet) describing these differences.
We mark the occurrences of $\tO\tL$ in the first $64$ letters of $\bt$:
\begin{equation}\label{eqn_TMtoB}
\begin{array}{c}
\begin{tikzpicture}[xscale=0.176,yscale=1,anchor=base,inner sep=1pt,
every node/.style={font=\small}]
\node (a0)	at (0,0)	{$\tO$};
\node (a1)	at (1,0)	{$\tL$};
\node (a2)	at (2,0)	{$\tL$};
\node (a3)	at (3,0)	{$\tO$};
\node (a4)	at (4,0)	{$\tL$};
\node (a5)	at (5,0)	{$\tO$};
\node (a6)	at (6,0)	{$\tO$};
\node (a7)	at (7,0)	{$\tL$};
\node (a8)	at (8,0)	{$\tL$};
\node (a9)	at (9,0)	{$\tO$};
\node (a10)	at (10,0)	{$\tO$};
\node (a11)	at (11,0)	{$\tL$};
\node (a12)	at (12,0)	{$\tO$};
\node (a13)	at (13,0)	{$\tL$};
\node (a14)	at (14,0)	{$\tL$};
\node (a15)	at (15,0)	{$\tO$};
\node (a16)	at (16,0)	{$\tL$};
\node (a17)	at (17,0)	{$\tO$};
\node (a18)	at (18,0)	{$\tO$};
\node (a19)	at (19,0)	{$\tL$};
\node (a20)	at (20,0)	{$\tO$};
\node (a21)	at (21,0)	{$\tL$};
\node (a22)	at (22,0)	{$\tL$};
\node (a23)	at (23,0)	{$\tO$};
\node (a24)	at (24,0)	{$\tO$};
\node (a25)	at (25,0)	{$\tL$};
\node (a26)	at (26,0)	{$\tL$};
\node (a27)	at (27,0)	{$\tO$};
\node (a28)	at (28,0)	{$\tL$};
\node (a29)	at (29,0)	{$\tO$};
\node (a30)	at (30,0)	{$\tO$};
\node (a31)	at (31,0)	{$\tL$};
\node (a32)	at (32,0)	{$\tL$};
\node (a33)	at (33,0)	{$\tO$};
\node (a34)	at (34,0)	{$\tO$};
\node (a35)	at (35,0)	{$\tL$};
\node (a36)	at (36,0)	{$\tO$};
\node (a37)	at (37,0)	{$\tL$};
\node (a38)	at (38,0)	{$\tL$};
\node (a39)	at (39,0)	{$\tO$};
\node (a40)	at (40,0)	{$\tO$};
\node (a41)	at (41,0)	{$\tL$};
\node (a42)	at (42,0)	{$\tL$};
\node (a43)	at (43,0)	{$\tO$};
\node (a44)	at (44,0)	{$\tL$};
\node (a45)	at (45,0)	{$\tO$};
\node (a46)	at (46,0)	{$\tO$};
\node (a47)	at (47,0)	{$\tL$};
\node (a48)	at (48,0)	{$\tO$};
\node (a48)	at (49,0)	{$\tL$};
\node (a50)	at (50,0)	{$\tL$};
\node (a51)	at (51,0)	{$\tO$};
\node (a52)	at (52,0)	{$\tL$};
\node (a53)	at (53,0)	{$\tO$};
\node (a54)	at (54,0)	{$\tO$};
\node (a55)	at (55,0)	{$\tL$};
\node (a56)	at (56,0)	{$\tL$};
\node (a57)	at (57,0)	{$\tO$};
\node (a58)	at (58,0)	{$\tO$};
\node (a59)	at (59,0)	{$\tL$};
\node (a60)	at (60,0)	{$\tO$};
\node (a61)	at (61,0)	{$\tL$};
\node (a62)	at (62,0)	{$\tL$};
\node (a63)	at (63,0)	{$\tO$};
\node (a64)	at (63.7,0)	{$,$};
\draw[>=latex,->] (a0.north)++(0,0.3)  -- ++(0,-0.3);
\draw[>=latex,->] (a0.north)++(3,0.3)  -- ++(0,-0.3);
\draw[>=latex,->] (a0.north)++(6,0.3)  -- ++(0,-0.3);
\draw[>=latex,->] (a0.north)++(10,0.3)  -- ++(0,-0.3);
\draw[>=latex,->] (a0.north)++(12,0.3)  -- ++(0,-0.3);
\draw[>=latex,->] (a0.north)++(15,0.3)  -- ++(0,-0.3);
\draw[>=latex,->] (a0.north)++(18,0.3)  -- ++(0,-0.3);
\draw[>=latex,->] (a0.north)++(20,0.3)  -- ++(0,-0.3);
\draw[>=latex,->] (a0.north)++(24,0.3)  -- ++(0,-0.3);
\draw[>=latex,->] (a0.north)++(27,0.3)  -- ++(0,-0.3);
\draw[>=latex,->] (a0.north)++(30,0.3)  -- ++(0,-0.3);
\draw[>=latex,->] (a0.north)++(34,0.3)  -- ++(0,-0.3);
\draw[>=latex,->] (a0.north)++(36,0.3)  -- ++(0,-0.3);
\draw[>=latex,->] (a0.north)++(40,0.3)  -- ++(0,-0.3);
\draw[>=latex,->] (a0.north)++(43,0.3)  -- ++(0,-0.3);
\draw[>=latex,->] (a0.north)++(46,0.3)  -- ++(0,-0.3);
\draw[>=latex,->] (a0.north)++(48,0.3)  -- ++(0,-0.3);
\draw[>=latex,->] (a0.north)++(51,0.3)  -- ++(0,-0.3);
\draw[>=latex,->] (a0.north)++(54,0.3)  -- ++(0,-0.3);
\draw[>=latex,->] (a0.north)++(58,0.3)  -- ++(0,-0.3);
\draw[>=latex,->] (a0.north)++(60,0.3)  -- ++(0,-0.3);
\draw[>=latex,->] (a0.north)++(63,0.3)  -- ++(0,-0.3);
\end{tikzpicture}
\end{array}
\end{equation}
from which we see that 

\[\diffseq=\mathtt{334233243342433233423}\cdots.\]
The blocks $\tO\tO\tO$ and $\tL\tL\tL$ do not appear as factors in $\bt$, since $\bt$ is a concatenation of the blocks $\tO\tL$ and $\tL\tO$.
Therefore the gaps between consecutive occurrences of $\tO\tL$ in $\bt$ are in fact bounded by $4$, and clearly they are bounded below by $2$ (since different occurrences of $\tO\tL$ cannot overlap).
It follows that we only need three letters, $\mathtt2$, $\mathtt3$, and $\mathtt4$, in order to capture the gap sequence.

The set of \emph{return words}~\cite{Durand1998,JustinVuillon2000} of a factor $w$ of $\bt$ is the set of words $x$ of the form $x=w\tilde x$, where 
$w$ is not a factor of $\tilde x$, and $w\tilde xw$ is a factor of $\bt$.
The gap sequence is the sequence of lengths of words in the decomposition
$\bt=x_0x_1\cdots$ of the Thue--Morse sequence into return words of $\tO\tL$,
which are $\tO\tL\tL$, $\tO\tL\tO$, $\tO\tL\tL\tO$, and $\tO\tL$ in order of appearance.

An appearance of the factor $\tO\tL$ marks the beginning of a block of $\tL$s in $\bt$.
Moreover, no other block of $\tL$s can appear before the next appearance of $\tO\tL$: between two blocks of $\tL$s we can find a block of one or more $\tO$s, and the last $\tO$ in this block is followed by $\tL$.
The assumption that we see a block of $\tL$s before the next appearance of $\tO\tL$ therefore leads to a contradiction. This argument is clearly visible in~\eqref{eqn_TMtoB}.
The sequence $\diffseq$ therefore gives the distances of consecutive blocks of $\tL$s.
We will see in Lemma~\ref{lem_diffseq_substitution} that the sequence $\diffseq$ is \emph{morphic}, or \emph{substitutive}.
That is, it can be described as the coding of a fixed point of a substitution over a finite alphabet.
Jeffrey Shallit (private communication, July 2019) 
proposed to prove the non-automaticity of $\diffseq$ to the author.
In the present paper, we investigate the sequence $\diffseq$ and the closely related, very well-known automatic sequence $\autoseq$ defined in Section~\ref{sec_aux}.
In particular, we prove the following theorem.
\begin{theorem}\label{thm_main}
Let $w$ be a factor of the Thue--Morse word of length at least $2$, and $C$ the
sequence of gaps between consecutive occurrences of $w$ in $\bt$. Then
$C$ is morphic, but not automatic.
\end{theorem}

Note that the set of positions where a given factor $w$ appears in $\bt$ \emph{is} $2$-automatic --- that is, its characteristic sequence is automatic.
This follows from the following theorem by Brown, Rampersad, Shallit, and Vasiga~\cite[Theorem~2.1]{BRSV2006}.
\begin{introtheorem}{A}
Let $\mathbf a=a_0a_1a_2\cdots$ be a $k$-automatic sequence over the alphabet $\Delta$, and let $w\in\Delta^\ast$. Then the set of of positions $p$ such that $w$ occurs beginning at position $p$ is $k$-automatic.
\end{introtheorem}

Concerning factors of length $1$, the corresponding gap sequence is automatic too; this follows from~\cite{BlanchetSadriCurrieRampersadFox2014}.

The second part of our paper is concerned with the \emph{discrepancy} of occurrences of $\tO\tL$-blocks in $\bt$.
More precisely, assume that $N$ is a nonnegative integer.
We count the number of times the factor $\tO\tL$ occurs in the first $N$ terms of the Thue--Morse sequence, and compare it to $N/3$:
\begin{equation}\label{eqn_def_discrepancy}
D_N\coloneqq\#\bigl\{0\leq n<N:\bt_n=\tO,\bt_{n+1}=\tL\bigr\}-\frac N3.
\end{equation}
From Theorem~A we can immediately derive that the sequence $(D_N)_{N\ge0}$ is $2$-\emph{regular}~\cite{AlloucheShallit1992,AlloucheShallit2003} as the sequence of partial sums of a $2$-automatic sequence:
the sequence having
\[
\begin{cases}
2/3&\mbox{if }\bt_n\bt_{n+1}=\tO\tL;\\
-1/3&\mbox{otherwise}
\end{cases}
\]
as its $n$th term is automatic as the sum of four $2$-automatic sequences,
and $D_N$ is the sum of the first $N$ terms of this sequence~\cite[Theorem~3.1]{AlloucheShallit1992}.
Our second theorem shows, more specifically, that $D_N$ can be obtained as the output sum of a base-$2$ transducer~(see Heuberger, Kropf, and Prodinger \cite{HeubergerKropfProdinger2015}, in particular Remark~3.10 in that paper).

\begin{theorem}\label{thm_discrepancy}
The sequence $(D_N)_{N\geq 0}$ is the sequence of output sums of a base-$2$ transducer.
In particular, $D_N\leq C\log N$ for some absolute implied constant $C$.
Moreover,
\begin{equation}\label{eqn_discrepancy_13Z}
\{D_N:N\geq 0\}=\frac 13\mathbb Z.
\end{equation}
\end{theorem}
Note that the unboundedness of $D_N$ follows from Corollary~4.10 in the paper~\cite{BertheBernales} by Berth\'e and Bernales on balancedness in words.

\bigskip\noindent
\textbf{Plan of the paper.}
In Section~\ref{sec_nonautomaticity} we prove that the gap sequence for a factor $w$ of $\bt$ is not automatic. The central step of this proof is the case $w=\tO\tL$, which will be handled in the first three sub-sections.
Section~\ref{sec_general} reduces the general case to this special case.
In Section~\ref{sec_explore} we study the automatic sequence $\autoseq$ on the three symbols $\{\ta,\tb,\tc\}$, closely related to the gap sequences. In particular, we lift this sequence to the seven-letter alphabet
$K=\{\ta,\Lbbleft,\Lbbright,\Lbleft,\Lbright,\Lcleft,\Lcright\}$.
From this new sequence we can in particular read off the discrepancy $D_N$ easily, which leads to a proof of Theorem~\ref{thm_discrepancy}.

\section{Proving the non-automaticity of gap sequences}\label{sec_nonautomaticity}
The main part of the proof of Theorem~\ref{thm_main} concerns non-automaticity of the gaps between occurrences of $\tO\tL$. As a second step in our proof, the general case will be reduced to this one.
\subsection{An auxiliary automatic sequence}\label{sec_aux}
We start by defining a substitution $\varphi$ on three letters:
\begin{equation}\label{eqn_varphi_morphism}
\begin{array}{llll}
\varphi:&\ta\mapsto \ta\tb\tc,&\tb\mapsto \ta\tc,&
\tc\mapsto\tb.
\end{array}
\end{equation}
The morphism $\varphi$ can be extended to $\{\ta,\tb,\tc\}^\mathbb N$ by concatenation, and we denote this extension by $\varphi$ again.
The unique fixed point (of length $>0$) of $\varphi$ is
\[    \autoseq=\ta\tb\tc\ta\tc\tb\ta\tb\tc\tb\ta\tc
      \ta\tb\tc\ta\tc\tb\ta\tc\ta\tb\tc\tb
      \ta\tb\tc\ta\tc\tb\ta\tb\tc\tb\ta\tc
      \ta\tb\tc\tb\ta\tb\tc\ta\tc\tb\ta\tc\cdots.    \]
This fixed point is a \emph{morphic}, or \emph{substitutive}, sequence~\cite[Chapter~7]{AlloucheShallitBook}.
As a fixed point of $\varphi$ --- without having to apply a coding of the fixed point --- it is even \emph{pure morphic}.
The sequence $\autoseq$ is in fact $2$-automatic, which follows from Berstel~\cite[Corollaire~4]{B1979}.
It is a `hidden automatic sequence' as treated very recently by Allouche, Dekking, and Queff\'elec~\cite{AlloucheDekkingQueffelec2020}. In fact, every automatic sequence can also be written as a coding of a fixed point of a non-uniform morphism~\cite{AS2019} and this sense is a `hidden' automatic sequence.
We will re-state a corresponding $2$-uniform substitution found by Berstel further down.
The sequence $\autoseq$, called \emph{ternary Thue--Morse sequence} (for example, in the OEIS~\cite[\texttt{A036577}]{OEIS}), \emph{Istrail squarefree sequence}~\cite{AlloucheDekkingQueffelec2020, Istrail1977}, or \texttt{vtm}~\cite{BlanchetSadriCurrieRampersadFox2014}, is well-known.
Citing Dekking~\cite{Dekking2016}, we note that it appears in fact twelve times on the OEIS~\cite{OEIS}, featuring all renamings of the letters corresponding to permutations of the sets $\{0,1,2\}$ and $\{1,2,3\}$.
These twelve entries are \texttt{A005679}, \texttt{A007413}, and \texttt{A036577}--\texttt{A036586}.
The sequence $\autoseq$ encodes the gaps between consecutive $\tL$s in $\bt$~\cite{BlanchetSadriCurrieRampersadFox2014}.
Thue~\cite{T1912} showed that $\autoseq$ is squarefree,
while Rao, Rigo, and Salimov~\cite{RaoRigoSalimov2015} later proved the stronger statement that $\autoseq$ even avoids \emph{$2$-binomial squares}~\cite{RigoSalimov2015}, thus settling in particular the question whether $2$-abelian squares are avoidable over a $3$-letter alphabet.
We will use the squarefreeness property in our proof of Theorem~\ref{thm_main}.
\begin{lemma}[Thue]\label{lem_squarefree}
The sequence $\autoseq$ is squarefree.
That is, no factor of the form $CC$, where $C$ is a finite word over $\{\ta,\tb,\tc\}$ of length at least $1$, appears in $\autoseq$.
\end{lemma}
We have the following important relation between $\autoseq$ and our problem.
\begin{lemma}\label{lem_correspondence}
The Thue--Morse sequence $\bt$ can be recovered from $\autoseq$ via the substitution
\begin{equation}\label{eqn_TM_abc_correspondence}
f:\ta\mapsto \tO\tL\tL\tO\tL\tO,\quad \tb\mapsto \tO\tL\tL\tO,\quad \tc\mapsto \tO\tL,
\end{equation}
by concatenation:
we have
\begin{equation}\label{eqn_TM_autoseq_concatenation}
\bt=f(\autoseq_0)f(\autoseq_1)\cdots.
\end{equation}
\end{lemma}
We will prove this in a moment.
From this observation, noting also that each of the three words $f(\ta)$, $f(\tb)$, and $f(\tc)$ begins with $\tO\tL$, we see that we can extract from $\autoseq$ the sequence of gaps between occurrences of the factor $\tO\tL$ in $\mathbf t$:
each $\ta$ yields two consecutive gaps of size $3$, each $\tb$ yields a gap of size $4$, and each $\tc$ a gap of size $2$.
\begin{proof}[Proof of Lemma~\ref{lem_correspondence}]
We prove that for $k\geq 1$,
\begin{equation}\label{eqn_autoseq_bt}
\bt_{[0,6\cdot 2^k]}=
f(\autoseq_0)f(\autoseq_1)\cdots f\bigl(\autoseq_{L_k}\bigr),
\end{equation}
where $L_k=3\cdot2^{k-1}-1$,
by induction. The case $k=1$ is just the trivial identity $\tO\tL\tL\tO\tL\tO\tO\tL\tL\tO\tO\tL=f(\ta)f(\tb)f(\tc)$.

Note that $\tau(f(\ta))=f(\ta)f(\tb)f(\tc)$, $\tau(f(\tb))=f(\ta)f(\tc)$, and
$\tau(f(\tc))=f(\tb)$. 
Extending $f$ by concatenation, for convenience of notation, to words over $\{\ta,\tb,\tc\}$, we obtain
$\tau(f(x))=f(\varphi(x))$ for each $x\in\{\ta,\tb,\tc\}$.
An application of the morphism $\tau$ to both sides of~\eqref{eqn_autoseq_bt} yields
\[
\begin{aligned}
\bt_{[0,6\cdot 2^{k+1})}
&=\tau\bigl(f(\autoseq_0)\bigr)\cdots\tau\bigl(f\bigl(\autoseq_{L_k}\bigr)\bigr)
=f(\varphi(\autoseq_0))\cdots f\bigl(\varphi\bigl(\autoseq_{L_k}\bigr)\bigr)
\\&=f\bigl(\varphi\bigl(\autoseq_0\cdots\autoseq_{L_k}\bigr)\bigr)
=f\bigl(\autoseq_1\cdots\autoseq_{L_{k+1}}\bigr)
=f(\autoseq_0)\cdots f\bigl(\autoseq_{L_{k+1}}\bigr).
\end{aligned}
\]
Note that we see by induction that $\varphi^k(\ta\tb\tc)$ contains each of the three letters $2^k$ times, hence the numbers $L_k$.
This proves the lemma.
\end{proof}
Since each of the words $f(\ta)$, $f(\tb)$, and $f(\tc)$ starts with $\tO\tL$,
the differences $a_j=k_{j+1}-k_j$ between successive occurrences of $\tO\tL$ in $\mathbf t$ are easily obtained from $\autoseq$ by the substitution
\begin{equation}\label{eqn_BA_transition}
r:\ta\mapsto \tM\tM,\quad\tb\mapsto \tQ,\quad\tc\mapsto \tZ.
\end{equation}
Here each $\ta$ yields two blocks $\tO\tL$, and each of $\tb$ and $\tc$ one block.

Let $\diffseq$ be the sequence of gaps between consecutive occurrences of $\tO\tL$ in the Thue--Morse sequence, and $\check{\diffseq}$ the corresponding sequence for $\tL\tO$.
\begin{lemma}\label{lem_diffseq_substitution}
The sequence $\diffseq$ is a morphic sequence, given by the substitution $\psi$ on the four letters $\ta,\tba,\tb,\tc$,
together with the coding $p$,
\begin{equation}\label{eqn_diffseq_morphism}
\begin{array}{lllll}
\psi:&\ta\mapsto\ta\tba,&\tba\mapsto\tb\tc,&\tb\mapsto\ta\tba\tc,&\tc\mapsto\tb,\\
p:&\ta\mapsto\tM,&\tba\mapsto\tM,&\tb\mapsto\tQ,&\tc\mapsto\tZ.
\end{array}
\end{equation}
Let $\bardiffseq$ denote the fixed point of $\psi$ starting with $\ta$.

The sequence $\check\diffseq$ is morphic.
More precisely, it is the image of $\bardiffseq$ under the morphism
\begin{equation}\label{eqn_morphic_image}
\check p:\ta\mapsto \tZ\tQ,\quad\tba\mapsto \tM\tM,\quad \tb\mapsto \tZ\tM\tM,\quad \tc\mapsto \tQ.
\end{equation}
\end{lemma}

Note that $\bardiffseq$ is the pointwise limit of the finite words $\psi^k(\ta)$, and begins as follows:
\[    \bardiffseq=\ta\tba\tb\tc\ta\tba\tc\tb\ta\tba\tb\tc\tb\ta\tba\tc
      \ta\tba\tb\tc\ta\tba\tc\tb\ta\tba\tc\ta\tba\tb\tc\tb
      \ta\tba\tb\tc\ta\tba\tc\tb\ta\tba\tb\tc\tb\ta\tba\tc
      \cdots.    \]

\begin{proof}

Let $q$ be the morphism that replaces $\ta$ by $\ta\tba$ and leaves $\tb$ and $\tc$ unchanged.
We show by induction on the length of a word $C$ over $\ta\tb\tc$ that
\begin{equation}\label{eqn_square}
q\bigl(\varphi(C)\bigr)=\psi(q(C)).
\end{equation}
This is clear for words of length $1$, since $q(\varphi(\ta))=\ta\tba\tb\tc=\psi(q(\ta))$,
$q(\varphi(\tb))=\ta\tba\tc=\psi(q(\tb))$, and $q(\varphi(\tc))=\tb=\psi(q(\tc))$.
Appending a letter $x\in\{\ta,\tb,\tc\}$ to a word $C$ for which the identity~\eqref{eqn_square} already holds, we obtain
\begin{equation*}
\begin{aligned}
q\bigl(\varphi(Cx)\bigr)&=q\bigl(\varphi(C)\varphi(x)\bigr)=q\bigl(\varphi(C)\bigr)q(\varphi(x))
\\&=\psi(q(C))\psi(q(x))=\psi(q(C)q(x))=\psi(q(Cx))
\end{aligned}
\end{equation*}
and therefore~\eqref{eqn_square} for $C$ replaced by $Cx$.
Next, we prove by induction, using~\eqref{eqn_square}, that
\[q(\varphi^k(\ta))=\psi^k(q(\ta)).\]
Clearly, this holds for $k=1$. For $k\geq 2$, we obtain
\[q\bigl(\varphi^k(\ta)\bigr)=
q\bigl(\varphi(\varphi^{k-1}(\ta))\bigr)
=\psi\bigl(q\bigl(\varphi^{k-1}(\ta)\bigr)\bigr)
=\psi\bigl(\psi^{k-1}(q(\ta))\bigr)
=\psi^k(q(\ta)).
\]
Noting that $q(\ta)=\psi(\ta)$ and $p\circ q=r$, the proof of the first part of Lemma~\ref{lem_diffseq_substitution} is complete.

We proceed to the second part, concerning $\check\diffseq$.
Note that by Corollary~7.7.5 in~\cite{AlloucheShallitBook} we only have to prove that $\check\diffseq=\check p\bigl(\bardiffseq)$.

Let
\begin{equation}\label{eqn_TM_checkabc_correspondence}
\check f:
\ta\mapsto \tO\tL\tL\tO\tL\tO,\quad
\tba\mapsto \tO\tL\tL\tO\tO\tL,\quad
\tb\mapsto \tO\tL\tL\tO\tL\tO\tO\tL,\quad
\tc\mapsto \tO\tL\tL\tO,
\end{equation}
and extend this function to words (finite or infinite) over $\{\ta,\tba,\tb,\tc\}$ by concatenation.

Applying $\tau$, we see by direct computation that
\[\tau\bigl(\check f(\ta)\bigr)=\tO\tL\tL\tO\tL\tO\tO\tL\tL\tO\tO\tL
=
\check f(\ta)\check f(\tba)
=\check f\bigl(\psi(\ta)\bigr),
\]
and analogously, we get
$\tau\bigl(\check f(x)\bigr)=\check f\bigl(\psi(x)\bigr)$
for each letter $x\in\{\tba,\tb,\tc\}$.
Applying this letter by letter, we obtain
\begin{equation}\label{eqn_symbol_swap}
\tau\bigl(\check f(w)\bigr)=\check f\bigl(\psi(w)\bigr)
\end{equation}
for every finite word over $\{\ta,\tba,\tb,\tc\}$.
By induction, we obtain
\[\tau^k\bigl(\check f(\ta)\bigr)=\check f\bigl(\psi^k(\ta)\bigr),\]
using the step
\[\tau^{k+1}\bigl(\check f(\ta)\bigr)=
\tau\bigl(\tau^k\bigl(\check f(\ta)\bigr)\bigr)
=\tau\bigl(\check f\bigl(\psi^k(\ta)\bigr)\bigr)
=\check f\bigl(\psi^{k+1}(\ta)\bigr).
\]
Noting that $\check f(\ta)$ begins with $\tO$, we obtain
$\bt=\check f\bigl(\bardiffseq)$.
In other words, the sequence $\bardiffseq$ yields the decomposition
$\bt=x_0x_1\cdots$ of the Thue--Morse sequence into return words of $\tO\tL\tL\tO$, where $x_j=\check f\bigl(\bardiffseq_j\bigr)$.
From this decomposition we can easily read off the sequence of gaps between occurrences of $\tL\tO$, since this word appears in each of the four return words, and the first occurrence always takes place at the same position, which is $2$.
In this way, we obtain the gaps $2$ and $3$ from the return word $\check f(\ta)$ each time $\ta$ appears in $\bardiffseq$.
Analogously, $\tba$ yields the gaps $3$ and $3$, the letter $\tb$ the gaps $2$, $3$, and $3$, and finally $\tc$ yields the gap $4$.
This proves the second part of Lemma~\ref{lem_diffseq_substitution}.
\end{proof}
\begin{remark}
A hint how to come up with the definition of $\psi$ can be found by
combining the substitutions $\varphi$ and $r$, given in~\eqref{eqn_varphi_morphism} and~\eqref{eqn_BA_transition} respectively, and considering the first few words $w_k=r(\varphi^k(\ta))$:
we have $w_1=\tM\tM\tQ\tZ$,
$w_2=\tM\tM\tQ\tZ\tM\tM\tZ\tQ$,
$w_3=\tM\tM\tQ\tZ\tM\tM\tZ\tQ\tM\tM\tQ\tZ\tQ\tM\tM\tZ$.
We see that a first guess for a definition of $\psi$, choosing $\tM\mapsto \tM\tM\tQ\tZ$, leads to the incorrect result $\tM\tM\tQ\tZ\tM\tM\tQ\tZ\cdots$ after the next iteration;
we are led to distinguishing between `the first letter ``$\tM$''' and `the second letter ``$\tM$''' in each occurrence of $\tM\tM$, which is exactly what our definition of $\psi$ does.
On the other hand, we directly obtain~\eqref{eqn_diffseq_morphism} by inspecting the decomposition of $\bt$ into return words of $\tO\tL$.
(Equivalently, we can study return words of $\tO\tL\tL\tO$, as we did in the second part of the proof of Lemma~\ref{lem_diffseq_substitution}.)
We can write the image under $\tau$ of each return word as a concatenation of return words, which yields the desired morphism.
\end{remark}

\subsection{Factors of $\diffseq$ appearing at positions in a residue class}
The main step in our proof of Theorem~\ref{thm_main} is given by the following proposition. For completeness, we let $\psi^0$ denote the identity, such that $\psi^0(w)=w$ for all words $w$ over $\{\ta,\tba,\tb,\tc\}$.
\begin{proposition}\label{prp_AP}
Let $\mu\ge0$ be an integer.
The sequence of indices where $\psi^{4\mu}(\ta)$ appears as a factor in $\bardiffseq$ has nonempty intersection with every residue class $a+m\mathbb Z$, where $m\ge1$ and $a$ are integers.
\end{proposition}
In the remainder of this section, we prove this proposition.
We work with the fourth iteration $\tpsi=\psi^4$ of the substitution $\psi$: we have
\begin{equation}\label{eqn_tpsi_onletters}
\begin{array}{l@{\hspace{0.4em}}l@{\hspace{1.5em}}l@{\hspace{0.4em}}l}
\tpsi(\ta)&=\ta\tba\tb\tc\ta\tba\tc\tb\ta\tba\tb\tc\tb\ta\tba\tc,&
\tpsi(\tba)&=\ta\tba\tb\tc\ta\tba\tc\tb\ta\tba\tc\ta\tba\tb\tc\tb,\\
\tpsi(\tb)&=\ta\tba\tb\tc\ta\tba\tc\tb\ta\tba\tb\tc\tb\ta\tba\tc\ta\tba\tb\tc\tb,&
\tpsi(\tc)&=\ta\tba\tb\tc\ta\tba\tc\tb\ta\tba\tc.
\end{array}
\end{equation}
We have the following explicit formulas for the lengths of $\tpsi^k(x)$, where $x\in\{\ta,\tba,\tb,\tc\}$:
\begin{equation}\label{eqn_lengths}
\begin{aligned}
a_k&\coloneqq \bigl\lvert\tpsi^k(\ta)\bigr\rvert=\bigl\lvert\tpsi^k(\tba)\bigr\rvert=16^k,\\
b_k&\coloneqq\bigl\lvert\tpsi^k(\tb)\bigr\rvert=\frac{4\cdot16^k-1}3,\\
c_k&\coloneqq\bigl\lvert\tpsi^k(\tc)\bigr\rvert=\frac{2\cdot16^k+1}3.
\end{aligned}
\end{equation}
The proof of this identity is based on the formula
\begin{equation*}
\left(\begin{matrix}4&4&4&4\\4&4&4&4\\5&5&6&5\\3&3&2&3\end{matrix}\right)^k
\!\!=\left(\begin{matrix}
16^k/4&16^k/4&16^k/4&16^k/4\\
16^k/4&16^k/4&16^k/4&16^k/4\\
(16^k-1)/3&(16^k-1)/3&(16^k+2)/3&(16^k-1)/3\\
(16^k+2)/6&(16^k+2)/6&(16^k-4)/6&(16^k+2)/6
\end{matrix}\right)\!,
\end{equation*}
valid for $k\geq 1$, which takes care of the numbers of the letters $\ta$, $\tba$, $\tb$, and $\tc$ in $\tpsi^k(\ta)$, $\tpsi^k(\tba)$, $\tpsi^k(\tb)$, and $\tpsi^k(\tc)$.
This formula can be proved easily by induction.
Moreover,~\eqref{eqn_lengths} also holds for $k=0$.

By applying $\tpsi^k$ on the first line of~\eqref{eqn_tpsi_onletters}, we see that each letter in $\{\ta,\tba,\tb,\tc\}$ is replaced by a word having the respective lengths $a_k$, $a_k$, $b_k$, $c_k$.
For each $\nu\geq 0$, it follows that the factor $\sigma^{\nu}(\ta)$, of length $a_{\nu}=16^{\nu}$,  can be found at the following positions
in $\tpsi^{\nu+1}(\ta)$:
\begin{equation}\label{eqn_steps}
\begin{array}{ll}
A^{(\nu,0)}\coloneqq 0,&
A^{(\nu,1)}\coloneqq 4\cdot16^{\nu},\\
A^{(\nu,2)}\coloneqq8\cdot16^{\nu},&
\displaystyle
A^{(\nu,3)}\coloneqq12\cdot16^{\nu}+\frac{4\cdot16^{\nu}-1}3.
\end{array}
\end{equation}
We may repeat this for $\nu-1,\nu-2,\ldots,\mu$, where $\mu\le\nu$ is a given natural number, from which we obtain the following statement.
For all integers $0\leq \mu\leq\nu$ and all
$\boldsymbol\varepsilon=\bigl(\varepsilon_{\mu},\varepsilon_{\mu+1},\ldots,\varepsilon_{\nu}\bigr)
\in\{0,1,2,3\}^{\nu-\mu+1}$,
the factor $\sigma^\mu(\ta)$ of length $16^\mu$ can be found at the position
\begin{equation}\label{eqn_Neps_def}
N_{\boldsymbol\varepsilon}\coloneqq A^{(\mu,\varepsilon_\mu)}+A^{(\mu+1,\varepsilon_{\mu+1})}+\cdots+A^{(\nu,\varepsilon_\nu)}
\end{equation}
in $\bardiffseq$.
There are other positions where the factor $\sigma^\mu(\ta)$ appears, but for our proof it is sufficient to consider these special positions.
We will show that we can find one among these indices $N_{\boldsymbol\varepsilon}$ in a given residue class $a+m\mathbb Z$.

Let us sketch the remainder of the proof.
The case that $m$ is even causes mild difficulties.
We therefore write $m=2^kd$, where $d$ is odd, and proceed in two steps.
As a first step, we will find integers $\mu$, $\nu$, and $\varepsilon_{\mu},\varepsilon_{\mu+1},\ldots,\varepsilon_{\lambda-1}\in\{0,1,2,3\}$, such that $N_{\varepsilon_\mu,\varepsilon_{\mu+1},\ldots,\varepsilon_{\lambda-1}}$ lies in any given residue class modulo $2^k$.
The second step will consist in refining the description by appending a sequence $(\varepsilon_\lambda,\ldots,\varepsilon_{\nu-1})\in\{0,1,2\}^{\nu-\lambda}$.
Since we exclude the digit $\varepsilon_i=3$, and
we will take care that
$16^\mu\geq 2^k$, we have
\[N_{\varepsilon_\mu,\ldots,\varepsilon_{\nu-1}}\equiv N_{\varepsilon_\mu,\ldots,\varepsilon_{\lambda-1}}\bmod 2^k.\]
We will choose the integers $\varepsilon_j$ for $\lambda\leq j<\mu$ in such a way that any given residue class modulo $d$ (note that $d$ is odd), is hit.
Due to the excluded digit $3$, this is a \emph{missing digit} problem,
and a short argument including exponential sums will finish this step.
Combining these two steps,
we will see that every residue class modulo $2^kd$ is reached. 
We will now go into the details.

\bigskip\noindent
\textbf{The first step: hitting a residue class modulo $2^k$}

\smallskip\noindent
We are interested in appearances of the initial segment $\sigma^\mu(\ta)$ in $\bardiffseq$ at positions lying in the residue class $a+2^k\mathbb Z$.
Let us assume in the following that
\begin{equation}\label{eqn_restriction}
16^\mu\geq 2^k.
\end{equation}
This lower bound on $\mu$ will not cause any problem.

We will choose $\lambda>\mu$ in a moment, and we set 
$\varepsilon_\mu=\cdots=\varepsilon_{\lambda-1}=3$.
Let us consider the integers $\alpha_0\coloneqq0$, and
for $1\leq \ell\leq \lambda-\mu$,
\[\alpha_\ell\coloneqq N_{\varepsilon_\mu,\ldots,\varepsilon_{\mu+\ell-1}}.\]
Assume that $0\leq \ell<\lambda-\mu$.
By~\eqref{eqn_Neps_def},~\eqref{eqn_restriction},
we have
\[
\begin{aligned}
\alpha_{\ell+1}-\alpha_\ell
&=12\cdot 16^{\mu+\ell}+\frac{4\cdot 16^{\mu+\ell}-1}3
\equiv \frac{4\cdot 16^{\mu+\ell}-1}3\bmod 2^\ell
\\&
\equiv \sum_{0\leq j\leq 2\mu+2\ell}4^j \bmod 2^\ell
\equiv \sum_{0\leq j<2\mu}4^j \bmod 2^\ell.
\end{aligned}
\]
The latter sum is an odd integer, and independent of $\ell$.
It follows that $(\alpha_\ell)_{0\leq \ell\leq \lambda-\mu}$ is an arithmetic progression modulo $2^k$, where the common difference is odd; choosing $\lambda\geq \mu+2^k$, we see that $(\alpha_\ell)_{0\leq \ell<\lambda-\mu}$ hits every residue class modulo $2^k$.
We summarise the first step in the following lemma.
\begin{lemma}\label{lem_first_step}
Let $k\geq 0$ and $\mu\geq k/4$, and choose
$\varepsilon_{\mu+\ell}=3$ for $\ell\ge0$.
The integers $N_{\varepsilon_\mu,\ldots,\varepsilon_{\lambda-1}}$ hit every residue class modulo $2^k$, as $\lambda$ runs through the integers $\geq\mu$.
\end{lemma}

\bigskip\noindent
\textbf{A discrete Cantor set --- missing digits}

\noindent
We follow the paper~\cite{EMS1998} by Erd\H{o}s, Mauduit, and S\'ark\"ozy, who studied integers with missing digits in residue classes.
Let $\mathcal W_\lambda$ be the set of nonnegative multiples of $16^\lambda$ having only the digits $0,4$, and $8$ in their base-$16$ expansion.
Set
\[U(\alpha)=\frac 13\sum_{0\leq k\leq 2}\e(4k\alpha)\quad\mbox{and}\quad
G(\alpha,\lambda,\nu)=\frac1{3^{\nu-\lambda}}\sum_{\substack{0\leq j<16^\nu\\j\in \mathcal W}}
\e(j\alpha),
\]
where $\e(x)=\exp(2\pi i x)$.
Note that elements $j\in\mathcal W_\lambda$ have the form
$j=\sum_{\lambda\leq k<\eta}4\,\varepsilon_k16^k$, where $\eta\geq 0$ and $\varepsilon_k\in\{0,1,2\}$ for $\lambda\leq k<\eta$.
In particular, $\mathcal W_\lambda\cap[0,16^\eta)$ has $3^{\eta-\lambda}$ elements for $\eta\geq \lambda$.
We obtain
\begin{equation}\label{eqn_G_prod}
\begin{aligned}
G(\alpha,\lambda,\nu)&=
\frac 1{3^{\nu-\lambda}}
\hspace{-1em}
\sum_{(\varepsilon_\lambda,\ldots,\varepsilon_{\nu-1})\in\{1,2,3\}^{\nu-\lambda}}
\hspace{-1em}
\e\bigl(4\,\varepsilon_\lambda 16^\lambda\alpha+\cdots+4\,\varepsilon_{\nu-1}16^{\nu-1}\alpha\bigr)
\\&=
\prod_{\lambda\leq r<\nu}
\frac 13\bigl(\e\bigl(0\cdot 16^r\alpha\bigr)+\e\bigl(4\cdot 16^r\alpha\bigr)+\e\bigl(8\cdot 16^r\alpha\bigr)\bigr)
\\&=
\prod_{\lambda\leq r<\nu}
U\left(16^r\alpha\right).
\end{aligned}
\end{equation}
The purpose of this section is to prove the following lemma.
\begin{lemma}\label{lem_step_2}
Let $\lambda\geq 0$ be an integer,
and $a,d$ integers such that $d\geq 1$ is odd.
Then
$\mathcal W_\lambda\cap\bigl(a+d\,\mathbb Z\bigr)$
contains infinitely many elements.
\end{lemma}
In order to prove this, we first show that it is sufficient to prove the following auxiliary result (compare~\cite[(4.3)]{EMS1998}).
\begin{lemma}\label{lem_EMS}
Assume that  $d\geq 1$ is an odd integer, and $\ell\in\{1,\ldots,d-1\}$.
Let $\lambda\geq 0$ be an integer.
Then
\[\lim_{\nu\rightarrow \infty}G\left(\frac \ell d,\lambda,\nu\right)=0.\]
\end{lemma}
In fact, by the orthogonality relation
\begin{equation*}
\frac 1d\sum_{0\leq n<d}\e(nk/d)=\Biggl\{\begin{array}{ll}
\!1,&\mbox{if }d\mid k;\\
\!0,&\mbox{otherwise},
\end{array}
\end{equation*}
we have
\begin{equation}\label{eqn_orthogonality_applied}
\begin{aligned}
\hspace{2em}&\hspace{-2em}
\frac1{3^{\nu-\lambda}}
\#\bigl\{
0\leq j<16^\nu:
j\in \mathcal W, j\equiv a\bmod d
\bigr\}
-\frac 1d
\\&=
\frac 1d
\frac1{3^{\nu-\lambda}}
\sum_{0\leq \ell<d}
\sum_{\substack{0\leq j<16^\nu\\j\in\mathcal W}}
\e\bigl(\ell(j-a)/d\bigr)
-\frac 1d\\
&=
\frac 1d
\sum_{1\leq \ell<d}
\e\bigl(-\ell a/d\bigr)
\frac1{3^{\nu-\lambda}}
\sum_{\substack{0\leq j<16^\nu\\j\in\mathcal W}}
\e\bigl(j\ell/d\bigr)
\leq
\sum_{1\leq \ell<d}
\left\lvert
G\left(\frac \ell d,\lambda,\nu\right)
\right\rvert.
\end{aligned}
\end{equation}
If $G\bigl(\ell/d,\lambda,\nu\bigr)$ converges to zero as $\nu\rightarrow\infty$, for all $\ell\in\{1,\ldots,d-1\}$, the last sum in~\eqref{eqn_orthogonality_applied} is eventually smaller than $1/d$.
This implies that the cardinalities $\#\bigl\{0\leq j<16^\nu:j\in\mathcal W_\lambda, j\equiv a\bmod d\bigr\}$ diverge to $\infty$, and therefore
$\mathcal W_\lambda\cap(a+d\mathbb Z)$ is infinite.

\begin{proof}[Proof of Lemma~\ref{lem_EMS}]
By~\eqref{eqn_orthogonality_applied}, we have to show that the product
\begin{equation}\label{eqn_prod_sufficient}
\prod_{\lambda\leq r<\nu}U(16^r\ell/d)
=
\prod_{\lambda\leq r<\nu}\bigl(1+\e\bigl(4\cdot 16^r\ell/d\bigr)+\e\bigl(8\cdot 16^r\ell/d\bigr)\bigr)
\end{equation}
converges to zero as $\nu\rightarrow\infty$.
To this end, we use the following lemma~\cite{D1972} by Delange.
\begin{lemma}[Delange]\label{lem_delange}
Assume that $q\geq 2$ is an integer and $z_1,\ldots,z_{q-1}$ are complex numbers such that $\lvert z_j\rvert \leq 1$ for $1\leq j<q$.
Then
\[\left\lvert \frac1q\bigl(1+z_1+\cdots+z_{q-1}\bigr)\right\rvert\leq 1-\frac1{2q}\max_{1\leq j<q}\bigl(1-\realpart z_j\bigr).\]
\end{lemma}

Since $d$ is odd and $1\leq \ell<d$, the integer
$4k16^r\ell$ is not a multiple of $d$ for $k\in\{1,2\}$.
It follows that $\realpart \e\bigl(4k16^r\ell/d\bigr)
\leq 1-\tilde\varepsilon$ for some $\tilde\varepsilon>0$ only depending on $d$.

Therefore each factor in~\eqref{eqn_G_prod} is smaller than $1-\varepsilon$, where $\varepsilon>0$ does not depend on $r$.
Consequently, by Lemma~\ref{lem_delange} the product~\eqref{eqn_prod_sufficient} converges to zero.
Lemma~\ref{lem_EMS}, and therefore Lemma~\ref{lem_step_2}, is proved.
\end{proof}
Now we combine the two steps, corresponding to the cases {\inparaenum[(i)] \item $2^k$ and \item $d$ odd.}

\noindent
Let $k\geq 0$ and $d\geq 1$ be integers, and $d$ odd. We are interested in a residue class $a+2^kd\,\mathbb Z$, where $a\in\mathbb Z$.
Choose
\[a^{(1)}\coloneqq a\bmod 2^k\in\{0,\ldots,2^k-1\}.\]
Choose $\mu$ large enough such that $16^\mu\geq 2^k$.
By Lemma~\ref{lem_first_step} there exists $\lambda\geq \mu$ in such a way that
\[\kappa^{(1)}\equiv a^{(1)}\bmod 2^k,\]
where $\kappa^{(1)}\coloneqq N_{\varepsilon_\mu,\ldots,\varepsilon_{\lambda-1}}$ and $\varepsilon_\ell=3$ for $\mu\leq\ell<\lambda$.
Next, choose
\[a^{(2)}\coloneqq \bigl(a-\kappa^{(1)}\bigr)\bmod d.\]
By Lemma~\ref{lem_step_2}, the set
$\mathcal W_\lambda\cap\bigl(a^{(2)}+d\,\mathbb Z\bigr)$
is not empty.
Let $\sum_{\lambda\leq \ell<\nu}4\,\varepsilon_\ell16^\ell$ be an element,
where $\varepsilon_\ell\in\{0,1,2\}$ for $\lambda\leq \ell<\nu$.
By~\eqref{eqn_Neps_def} we have
\begin{align*}
\kappa\coloneqq N_{\varepsilon_\mu,\ldots,\varepsilon_{\lambda-1},\varepsilon_\lambda,\ldots,\varepsilon_{\nu-1}}
&=
\kappa^{(1)}+\kappa^{(2)},
\end{align*}
where
\begin{align*}
\kappa^{(2)}\coloneqq
N_{\varepsilon_\lambda,\ldots,\varepsilon_{\nu-1}}.
\end{align*}
The integer $\kappa^{(1)}$ lies in the residue class $a^{(1)}+2^k\mathbb Z$ by construction, while $\kappa^{(2)}$ is divisible by $2^k$, as no digit among $\varepsilon_\lambda,\ldots,\varepsilon_{\nu-1}$ equals $3$.
It follows that $\kappa\in a^{(1)}+2^k\mathbb Z=a+2^k\mathbb Z$.
Moreover, by~\eqref{eqn_steps},
\[\kappa^{(2)}=\sum_{\lambda\leq \ell<\nu}4\,\varepsilon_\ell16^\ell\in a^{(2)}+d\,\mathbb Z,\]
hence
$\kappa=\kappa^{(1)}+\kappa^{(2)}
\equiv \kappa^{(1)}+\bigl(a-\kappa^{(1)}\bigr)\equiv a\bmod d$.

Summarising, we have
$\kappa\in \bigl(a+2^k\mathbb Z\bigr)\cap\bigl(a+d\,\mathbb Z\bigr)$.
Since $2^k$ and $d$ are coprime, which implies $2^k\mathbb Z\cap d\,\mathbb Z=2^kd\,\mathbb Z$, we have
$\bigl(a+2^k\mathbb Z\bigr)\cap\bigl(a+d\,\mathbb Z\bigr)=a+2^kd\,\mathbb Z$ 
(applying a shift by $a$)
and therefore $\kappa\in a+2^kd\,\mathbb Z$.
This finishes the proof of Proposition~\ref{prp_AP}.\qed

\subsection{Non-automaticity of $\mathbf B$}
In order to prove that $\mathbf B$ is not automatic, we use the characterization by the $k$-kernel:
a sequence $(a_n)_{n\geq 0}$ is $k$-automatic if and only if the set
\begin{equation}\label{eqn_k_kernel}
\bigl\{\bigl(a_{\ell+k^jn}\bigr)_{n\geq 0}: j\geq 0, 0\leq \ell<k^j\bigr\}
\end{equation}
is finite.

We are now in the position to prove that \emph{any} two arithmetic subsequences of $\diffseq$ with the same modulus $m$ and different shifts $\ell_1$, $\ell_2$ are different:
the sequences $\bigl(\diffseq(\ell_1+nm)\bigr)_{n\geq 0}$ and $\bigl(\diffseq(\ell_2+nm)\bigr)_{n\geq 0}$ cannot be equal.
This will prove in particular that the $k$-kernel is infinite and thus non-automaticity of the gap sequence for $\tO\tL$.

Let us assume, in order to obtain a contradiction, that the sequence $\diffseq$ contains two identical arithmetic subsequences with common differences equal to $m$, indexed by
$n\mapsto \ell_1+nm$ and $n\mapsto \ell_2+nm$ respectively,
where $\ell_1<\ell_2$.
Let $r=\ell_2-\ell_1$, and choose $\mu$ large enough such that $16^\mu\geq 2r$.
By Proposition~\ref{prp_AP}, the block $\sigma^\mu(\ta)$ appears in $\bardiffseq$ at positions that hit each residue class.
In particular, for each $s\in\{0,\ldots,r-1\}$ we choose the residue class $\ell_1-s+m\mathbb Z$, and we can find an index $n$ such that
$\sigma^\mu(\ta)$ appears at position $\ell_1-s+nm$ in $\bardiffseq$.
Since $16^m\geq 2r>s$, this means that $\ell_1+mn$ hits the $s$th letter in $\sigma^\mu(\ta)$, in symbols,
\[\bardiffseq_{\ell_1+nm}=\sigma^\mu(\ta)\bigr\rvert_s.\]
Since $s+r$ is still in the range $[0,16^\mu)$, we also have
\[\bardiffseq_{\ell_2+nm}=\sigma^\mu(\ta)\bigr\rvert_{s+r}\]
for the same index $n$.
Applying the coding $p$ defined in~\eqref{eqn_diffseq_morphism}, and our equality assumption, we see that
\[
\diffseq_s=
p\bigl(\sigma^\mu(\ta)\bigr\rvert_s\bigr)=
\diffseq_{\ell_1+nm}=\diffseq_{\ell_2+nm}=
p\bigl(\sigma^\mu(\ta)\bigr\rvert_{s+r}\bigr)=\diffseq_{s+r}.\]
Carrying this out for all $s\in\{0,\ldots,r-1\}$, we see that
the first $2r$ terms of $\diffseq$ form a square.
Now there are two cases to consider.

\noindent\textbf{The case $r=1$.}
Assume that $\diffseq_{\ell_1+nm}=\diffseq_{\ell_1+1+nm}$ for all $n\geq 0$.
By Proposition~\ref{prp_AP}, the positions where the prefix $\mathtt{3342}=\diffseq_0\diffseq_1\diffseq_2\diffseq_3$ appears as a factor in $\diffseq$ hit every residue class.
In particular, there is an index $n$ such that 
the block $\mathtt{3342}$ can be found at position $\ell_1-1+nm$ in $\diffseq$.
This implies
$\mathtt{3}=\diffseq_1=\diffseq_{\ell_1+nm}=\diffseq_{\ell_1+1+nm}=\diffseq_2=\mathtt{4}$,
a contradiction.

\noindent\textbf{The case $r\ge2$.}
In this case we will resort to the fact, proved below, that $\diffseq$ does not contain squares of length $>2$.
Therefore we get a contradiction also in this case.
In order to complete the proof that $\diffseq$ is not automatic, it remains to prove (the second part of) the following result.
\begin{lemma}
The infinite word $\bardiffseq$ is squarefree.
The word $\diffseq$ does not contain squares of length $>2$.
\end{lemma}
\begin{proof}
We begin with the first statement. 
Note first that, by the morphism~\eqref{eqn_BA_transition},
letters `$\tM$' in $\diffseq$ appear in pairs;
moreover, the squarefreeness of $\autoseq$ implies that there are no runs of three or more $\tM$s.
This implies that the morphism $r$ defined in~\eqref{eqn_BA_transition} can be `reversed' in the sense that $\autoseq$ can be restored from $\diffseq$ by the (unambiguous) rule $\tilde r:\tM\tM\mapsto\ta$, $\tQ\mapsto\tb$, $\tZ\mapsto\tc$.
Also, $\bardiffseq$ can be restored from $\diffseq$ by the (unambiguous) rule $\tM\tM\mapsto\ta\tba$, $\tQ\mapsto\tb$, $\tZ\mapsto\tc$, thus reversing the effect on $\bardiffseq$ of the morphism $p$ defined in~\eqref{eqn_diffseq_morphism}.
In particular, since $\autoseq$ is squarefree, each occurrence of the factor $\ta\tba$ in $\bardiffseq$ is bordered by symbols $\in\{\tb,\tc\}$ (where of course the first occurrence at $0$ is not bordered on the left by another symbol).

Assume, in order to obtain a contradiction, that the square $CC$ is a factor of $\bardiffseq$.
We distinguish between two cases.

\noindent\textbf{The case $\lvert C\rvert=1$.}
Let $C$ consist of a single symbol $x\in\{\ta,\tba,\tb,\tc\}$. The squarefreeness of $\autoseq$ forbids $x\in\{\tb,\tc\}$; moreover, we saw a moment ago that $\ta$ and $\tba$ may only appear together, bordered by symbols $\in\{\tb,\tc\}$. This excludes the possibility $x\in\{\ta,\tba\}$, therefore this case leads to a contradiction.

\noindent\textbf{The case $\lvert C\rvert\geq 2$.}
There are two cases to consider.
{\inparaenum[(i)] \item
Assume that $C$ begins with $\tba$.
In this case, $C$ has to end with $\ta$: the concatenation $CC$ has to be a factor of $\bardiffseq$, and therefore the symbol $\tba$ at the start of the second `$C$' has to be preceded by a symbol $\ta$.
Analogously, each occurrence of the word $CC$ is immediately preceded by $\ta$, and followed by $\tba$. That is, $\ta CC\tba$ appears as a factor of $\bardiffseq$. Writing $C=\tba y\ta$ for a finite (possibly empty) word $y$ over $\{\ta,\tba,\tb,\tc\}$, we see that
$\ta\tba y\ta\tba y\ta\tba$ is a factor of $\bardiffseq$.
Applying the coding $p$, it follows that $T=\ta\ta y\ta\ta y\ta\ta$ appears in $\diffseq$, and it is a concatenation of the words $\tM\tM$, $\tQ$, and $\tZ$.
Consequently, it makes sense to
apply the `inverse morphism' $\tilde r:\tM\tM\mapsto \ta$, $\tQ\mapsto\tb$, $\tZ\mapsto\tc$.
Therefore $\tilde r(T)=\ta z\ta z\ta$, for some finite word $z$ over $\{\ta,\tb,\tc\}$, appears in $\autoseq$. 
This contradicts Lemma~\ref{lem_squarefree}.
\item Assume that $C$ starts with a letter $\in\{\ta,\tb,\tc\}$.
In this case, $C$ ends with a letter $\in\{\tba,\tb,\tc\}$: otherwise, the concatenation $CC$, and therefore $\bardiffseq$, would contain $\ta\ta$, which we have already ruled out.
We apply $p$, and in this case $p(C)$ is a concatenation of the words $\tM\tM$, $\tQ$, and $\tZ$.
Therefore we can form $\tilde r(p(C))$, revealing that the square $\tilde r(p(C))\tilde r(p(C))$ is a factor of $\autoseq$. This is a contradiction.
}

We have to prove the second statement. Assume that $CC$ is a factor of $\diffseq$, where $\lvert C\rvert\geq 2$.
This proof is analogous to the corresponding case for $\bardiffseq$, and we skip some of the details that we have already seen there.
{\inparaenum[(i)] \item
Assume that $C$ begins with exactly one $\ta$.
In this case, $C$ has to end with exactly one $\ta$,
and therefore $C=\ta y \ta$ for a finite word $y$ over $\{\ta,\tb,\tc\}$.
It follows that $\ta\ta y\ta\ta y\ta\ta$ is a factor of $\diffseq$.
Applying $\tilde r$, we obtain a contradiction to Lemma~\ref{lem_squarefree}.

\item Assume that $C$ starts with $\ta\ta$, $\tb$, or $\tc$.
In this case, $C$ ends with $\ta\ta$, $\tb$, or $\tc$, otherwise $CC$, and therefore $\bardiffseq$, would contain a block of $\ta$s of length $\neq 2$.
We apply $p$ on the word $CC$, followed by $\tilde r$, which yields the square $\tilde r(p(C))\tilde r(p(C))$. Again, this contradicts Lemma~\ref{lem_squarefree}.
}
\end{proof}

Summarising, arithmetic subsequences of $\diffseq$ with common difference $m$ are distinct as soon as their offsets differ.
In particular, for each integer $k\ge2$ the $k$-kernel of $\diffseq$ is infinite.
Therefore $\diffseq$ is not automatic, which proves the case $w=\tO\tL$ of Theorem~\ref{thm_main}.

\subsection{Occurrences of general factors in $\bt$}\label{sec_general}
We begin with the case $w=\tL\tO$.
We will work with the Thue--Morse morphism $\tau:\tO\mapsto\tO\tL$, $\tL\mapsto\tL\tO$, defined in~\eqref{eqn_tau}.
First of all, we recall the well-known fact that $a_{k+1}=\tau^{k+1}(\tO)$ can be constructed from $a_k=\tau^k(\tO)$ by
concatenating $a_k$ and its Boolean complement $\overline{a_k}$ (which replaces each $\tO$ by $\tL$ and each $\tL$ by $\tO$).
The proof of this little fact is by an easy induction.
For $k=0$ we have $a_1=\tO\tL=\tO\overline{\tO}$.
The case $k\geq 1$ makes use of the identity $\tau(\overline w)=\overline{\tau(w)}$, valid for each word $w$ over $\{\tO,\tL\}$, which follows from the special structure of the morphism $\tau$.
Applying this identity and the induction hypothesis, we obtain
\begin{align*}
a_{k+1}&=\tau\bigl(\tau^k(\tO)\bigr)=\tau\bigl(a_{k-1}\overline{a_{k-1}}\bigr)
\\&=\tau\bigl(a_{k-1}\bigr)\tau\bigl(\overline{a_{k-1}}\bigr)
=\tau\bigl(a_{k-1}\bigr)\overline{\tau\bigl(a_{k-1}\bigr)}
=a_k\overline{a_k}.
\end{align*}

Using this, we show that for even $k\geq 0$, the word $\tau^k(\tO)$ is a palindrome.
The case $k=0$ is trivial.
If $a_k=\tau^k(\tO)$ is a palindrome, then $a_{k+2}=\tau(\tau(a_k))=\tau(a_k\overline{a_k})=a_k\overline{a_k}\overline{a_k}a_k$ is clearly a palindrome too, and the statement follows by induction.
In particular, we see from the above that
\begin{equation}\label{eqn_block_pairs}
\tau^k(\tO)=\tau^{k-1}(\tO)\tau^{k-1}(\tL),\quad
\tau^k(\tL)=\tau^{k-1}(\tL)\tau^{k-1}(\tO)
\quad\mbox{ for all }k\geq 0.
\end{equation}
Note that, by applying $\tau^k$ on $\bt$, every $\tO$ gets replaced by $\tau^k(\tO)$ and every $\tL$ by $\tau^k(\tL)$, and the result is again $\bt$ since it is a fixed point of $\tau$.
It follows that
\begin{equation}\label{eqn_concat}
\begin{array}{rl}
\mbox{for all }k\geq 0,&\mbox{ we have }
\bt=A_{\bt_0}A_{\bt_1}A_{\bt_2}\cdots,\\[1mm]
&\mbox{ where }
A_x=\tau^k(x)\mbox{ for }x\in\{\tO,\tL\}.
\end{array}
\end{equation}

Let $(r_j)_{j\geq 0}$ be the increasing sequence of indices where $\tL\tO$ occurs in $\bt$.
For $k$ even, let $J=J(k)$ be the number of occurrences of $\tL\tO$ with indices $\leq 2^k-2$.
Note that $r_{J-1}=2^k-2$.
We read the (palindromic) sequence $a_k$, of length $2^k$, backwards;
it follows that $(2^k-1-r_{J-1-j})_{0\leq j<J}$ is the increasing sequence of indices pointing to the letter $\tL$ in an occurrence of $\tO\tL$ in $a_k$.
Therefore
\[\bigl(2^k-2-r_{J-1-j}\bigr)_{0\leq j<J}\]
is the increasing sequence of indices where $\tO\tL$ occurs in $a_k$.
Consequently, by the definition of $\diffseq$ as the differences of these indices, we obtain
$\diffseq_j=-r_{J-1-(j+1)}+r_{J-1-j}$ for $0\leq j<J-1$ and thus
\begin{equation}\label{eqn_backwards}
r_{j+1}-r_j=\diffseq_{J-2-j}\quad\mbox{for }0\leq j\le J-2.
\end{equation}
We have to prove that the sequence
\begin{equation}\label{eqn_checkdiffseq_def}
\check{\diffseq}=(r_{j+1}-r_j)_{j\geq 0}
\end{equation}
is not automatic.
More generally, we prove that any two arithmetic subsequences
\[L^{(1)}=\bigl(\check{\diffseq}(\ell_1+nd)\bigr)_{n\geq 0},\quad
L^{(2)}=\bigl(\check{\diffseq}(\ell_2+nd)\bigr)_{n\geq 0},\]
where $d\geq 1$ and $\ell_1\neq\ell_2$, are different.
In order to obtain a contradiction, let us assume that $L^{(1)}=L^{(2)}$,
and let $k\geq 0$ be even.
By~\eqref{eqn_backwards}, we get arithmetic subsequences $M_1$, $M_2$ of $\diffseq$ with common difference $d$, different offsets $m_1(k),m_2(k)\in\{0,\ldots,d-1\}$, and length equal to $J(k)-1$, such that
\[\diffseq_{m_1(k)+nd}=M^{(1)}_j=M^{(2)}_j=\diffseq_{m_2(k)+nd}\quad\mbox{for }0\leq n\le J(k)-2.\]
Note the important fact that the offsets $m_j(k)$ are bounded by $d$.
Since there are only $d(d-1)/2$ pairs $(a,b)\in \{0,\ldots,d-1\}^2$ with $a\neq b$,
it follows that there are two different offsets $0\le\overline{m_1},\overline{m_2}<d$ with the following property:
there are arbitrarily long arithmetic subsequences of $\diffseq$ with indices of the form $\overline{m_1}+nd$ and $\overline{m_2}+nd$ respectively, taking the same values.
This is just the statement that the \emph{infinite} sequences $\bigl(\diffseq_{\overline{m_1}+nd})_{n\geq 0}$ and $\bigl(\diffseq_{\overline{m_2}+nd}\bigr)_{n\geq 0}$ are equal.
In the course of proving that $\diffseq$ is not automatic (which is the case $w=\tO\tL$ of Theorem~\ref{thm_main}) we proved that this is impossible, and we obtain a contradiction. The sequence $\check{\diffseq}$ is therefore not automatic either, which finishes the case $w=\tL\tO$.

We proceed to the case $w=\tO\tO$.
Let $(a_i)_{i\geq 0}$ be the increasing sequence of indices $j$ such that $\bt_j\bt_{j+1}=\tO\tO$.
Assume that $i\geq 0$, and set $j\coloneqq a_i$.
We have $j\equiv 1\bmod 2$, since $\bt_{2j'}=\overline{\bt_{2j'+1}}$ for all $j'\geq 0$ (where the overline denotes the Boolean complement, $\tO\mapsto \tL$, $\tL\mapsto \tO$). Equality $\bt_j=\bt_{j+1}$ (as needed) can therefore only occur at odd indices $j$, and we choose $j'\ge0$ such that $j=2j'+1$.
Necessarily, $\bt_{j'}=\tL$ and $\bt_{j'+1}=\tO$, since the identities $\bt_{2j'+1}=\overline{\bt_{j'}}$ and $\bt_{2j'+2}=\bt_{j'+1}$ would produce an output $\bt_{2j'+1}\bt_{2j'+2}\neq\tO\tO$ in the other case.
On the other hand, $\bt_{j'}\bt_{j'+1}=\tL\tO$ indeed implies $\bt_{2j'+1}\bt_{2j'+2}=\tO\tO$.
Each occurrence of $\tO\tO$ in $\bt$, at position $j$, therefore corresponds in a bijective manner to an occurrence of $\tL\tO$, at position $(j-1)/2$ (which is an integer).
It follows that the corresponding gap sequence equals $2\check\diffseq$, which is not automatic by the already proved case $w=\tL\tO$.

In a completely analogous manner, we can reduce the case $w=\tL\tL$ to the case $\tO\tL$,
and the gap sequence equals $2\diffseq$, which is not automatic either.

We will now reduce the case of general factors $w$ of $\bt$ of length $\ge3$ to these four cases.
\begin{lemma}\label{lem_alternating}
For $x,y\in\{\tO,\tL\}$, let $(a^{xy}_k)_{k\ge0}$ be the increasing sequence of indices $j$ such that $\bt_j\bt_{j+1}=xy$.
We have
\begin{equation}\label{eqn_chains}
\begin{aligned}
&a^{\tO\tL}_0<a^{\tL\tO}_0<a^{\tO\tL}_1<a^{\tL\tO}_1<a^{\tO\tL}_2<a^{\tL\tO}_2<
\cdots\quad\mbox{and}\\
&a^{\tL\tL}_0<a^{\tO\tO}_0<a^{\tL\tL}_1<a^{\tO\tO}_1<a^{\tL\tL}_2<a^{\tO\tO}_2<\cdots.
\end{aligned}
\end{equation}
\end{lemma}
\begin{proof}
First of all, $\bt$ begins with $\tO\tL\tL$, whence the first items of the two displayed chains of inequalities.
The first chain is almost trivial since after each block of consecutive $\tO$s, a letter $\tL$ follows, and vice versa.

Let us prove the second series of inequalities by induction.
Assume that
$a^{(\tL\tL)}_0<a^{(\tO\tO)}_0<\cdots<a^{(\tO\tO)}_{i-1}<a^{(\tL\tL)}_i=j$. Then $\bt_j\bt_{j+1}=\tL\tL$, and it follows that $\bt_{j+2}=\tO$, since $\tL\tL\tL$ is not a factor of $\bt$.
Two cases can occur.
{\inparaenum[(i)] \item If $\bt_{j+3}=\tO$, then clearly
$a^{(\tL\tL)}_i<a^{(\tO\tO)}_i=j+2$  by our hypothesis.
\item Otherwise, we have $\bt_j\bt_{j+1}\bt_{j+2}\bt_{j+3}=\tL\tL\tO\tL$.
Necessarily, $j$ is odd: if $j=2j'$, it would follow that $\bt_j\bt_{j+1}\in\{\tO\tL,\tL\tO\}$, but we need $\tL\tL$.
Moreover, $j\equiv 3\bmod 4$ is also not possible: Let $j+1=4j'$. Then $\bt_{j+1}\bt_{j+2}\bt_{j+3}\in\{\tO\tL\tL,\tL\tO\tO\}$, but we need $\tL\tO\tL$.
It follows that $j\equiv 1\bmod 4$, and therefore $\bt_{j+4}\bt_{j+5}=\tO\tO$, which implies $a^{(\tO\tO)}_i=j+4$.
}
By a completely analogous argument (reversing the roles of $\tL$ and $\tO$),
we may finish the proof of Lemma~\ref{lem_alternating} by induction.
\end{proof}

Let $w$ be a factor of $\bt$, of length $\ge3$.
Choose $k\geq 0$ minimal such that $w$ is a factor of some
$a^{xy}_k=\tau^k(x)\tau^k(y)$, where $x,y\in\{\tO,\tL\}$.
By minimality, $w$ is not a factor of $\tau^k(\tO)$ or $\tau^k(\tL)$, using~\eqref{eqn_block_pairs}.
Consequently, $w$ appears at most once in each $a^{xy}_k$.
Next, we need the fact that $\bt$ is \emph{overlap-free}~\cite{B1995,BRSV2006,T1912}, meaning that is does not contain a factor of the form $axaxa$, where $a\in\{\tO,\tL\}$ and $x\in\{\tO,\tL\}^\ast$.
We derive from this property that $w$ cannot occur simultaneously in both members of either of the pairs
\begin{equation}\label{eqn_pairs}
(a^{\tO\tO}_k,a^{\tO\tL}_k),\quad(a^{\tO\tO}_k,a^{\tL\tO}_k),\quad
(a^{\tL\tL}_k,a^{\tO\tL}_k),\quad(a^{\tL\tL}_k,a^{\tL\tO}_k).
\end{equation}
For example, assume that $w$ is a factor of both $a^{\tO\tO}_k$ and
$a^{\tO\tL}_k$.
By minimality, as we had before,
\[
\tau^k(\tO)\tau^k(\tO)=AwB,\quad
\tau^k(\tO)\tau^k(\tL)=A'wB',
\]
where $A$ and $A'$ are initial segments of $\tau^k(\tO)$, and $B$ resp. $B'$ are final segments of $\tau^k(\tO)$ resp. $\tau^k(\tL)$, and all of these segments are proper subwords of the respective words.
We have $A\neq A'$, since otherwise $\tau^k(\tO)=\tilde wB=\tilde wB'=\tau^k(\tL)$ for some $\tilde w$ that is not the empty word. This contradicts the fact that $\tau^k(\tO)=\overline{\tau^k(\tL)}$.
Let us, without loss of generality, assume that $\lvert A\rvert<\lvert A'\rvert$.
The first $2^k$ letters of $Aw$ and $A'w$ are equal, in symbols,
\begin{equation}\label{eqn_initial_equal}
(Aw)\bigr\rvert_{[0,2^k)}=(A'w)\bigr\rvert_{[0,2^k)}.
\end{equation}

We can therefore choose $a\in\{\tO,\tL\}$ and $w_1,w_2\in\{\tO,\tL\}^\ast$ in such a way that
$aw_1w_2=w$ and $Aaw_1=A'$.
Then trivially $Aw=Aaw_1w_2=A'w_2$, and since $\lvert A\rvert<2^k$, $\lvert A'\rvert<2^k$, it follows from~\eqref{eqn_initial_equal} that $w_2=aw_3$ for some $w_3\in\{\tO,\tL\}^\ast$.
Finally, the factor $A'w$ of $\bt$ can be written as $A'w=Aaw_1w=Aaw_1aw_1w_2=Aaw_1aw_1aw_3$, which contradicts the overlap-freeness of $\bt$.
The other three cases, corresponding to the second through fourth pairs in~\eqref{eqn_pairs}, are analogous.
We have therefore shown that the set of $A\in\{a^{\tO\tO}_k,a^{\tO\tL}_k,a^{\tL\tO}_k,a^{\tL\tL}_k\}$ such that $w$ is a factor of $A$ is a subset of either $\{a^{\tO\tL}_k,a^{\tL\tO}_k\}$ or $\{a^{\tO\tO}_k,a^{\tL\tL}_k\}$.

\bigskip\noindent\textbf{First case.}
Let $w$ be a factor of $a^{\tO\tL}_k$, or of $a^{\tL\tO}_k$.
Assume first that $w$ is a factor of $a^{\tO\tL}_k$, but not of $a^{\tL\tO}_k$.
In this case, we show that the gap sequence for $w$ is given by the gap sequence for $a^{\tO\tL}_k$:
{\inparaenum[(i)] \item each occurrence of $a^{\tO\tL}_k$ yields exactly one occurrence of $w$ (involving a constant shift); \item by~\eqref{eqn_concat}, every occurrence of $w$ takes place within a block of the form $a^{xy}_k$; \item only the block $a^{\tO\tL}_k$ is eligible.}
We prove that $a^{\tO\tL}_k$ appears exactly at positions $2^kj$ in $\bt$, where $\bt_j\bt_{j+1}=\tO\tL$. The easy direction follows from~\eqref{eqn_concat}: each occurrence of $\tO\tL$ yields an occurrence of $a^{\tO\tL}_k$, where the index has to be multiplied by $2^k$.
On the other hand, it is sufficient to show that $a^{\tO\tL}_k$ can only appear on positions $2^kj$. Given this, there is no admissible choice for $(\bt_j,\bt_{j+1})$ different from $(\tO,\tL)$, by \eqref{eqn_concat}.
Suppose that we already know this for some $k\ge0$ (the case $k=0$ being trivial).
Assume that
\begin{equation}\label{eqn_assumption}
a^{\tO\tL}_{k+1}=\tau^k(\tO)\tau^k(\tL)\tau^k(\tL)\tau^k(\tO)
\mbox{ appears on some position }\ell.
\end{equation}
Since $\tau^k(\tO)\tau^k(\tL)=a^{\tO\tL}_k$, we know by hypothesis that $\ell\equiv 0\bmod 2^k$.
Assume that the case $\ell\equiv 2^k\bmod 2^{k+1}$ occurs.
We set $\ell=(2j+1)2^k$ for some $j\geq 0$.
Our assumption~\eqref{eqn_assumption} implies $\tau^k(\tL)=\tau^k(\bt_{2j+2})=\tau^k(\bt_{j+1})$ and therefore $\bt_{j+1}=\tL$, which implies that $\tau^{k+1}(\bt_{j+1})=\tau^k(\tL)\tau^k(\tO)$ appears on position $\ell+2^k=(2j+2)2^k$ in $\bt$. This is incompatible with~\eqref{eqn_assumption}.
In particular, the gap sequence for $w$, which is identical to the gap sequence for $a^{\tO\tL}_k$, is given by $2^k\diffseq$, and therefore not automatic.
Switching the roles of $\tO$ and $\tL$ in this proof, we also obtain non-automaticity for the case that $w$ is a factor of $a_k^{\tL\tO}$, but not of $a_k^{\tO\tL}$ --- with the sequence $2^k\check\diffseq$ as the corresponding gap sequence.

Let $w$ be a factor of both $a^{\tO\tL}_k$ and $a^{\tL\tO}_k$.
In this case, each occurrence of $w$ in $\bt$ takes place within a subblock of $\bt$ of one of these two forms.
By Lemma~\ref{lem_alternating}, combined with the above argument that occurrences of $a^{\tO\tL}_k$ resp. $a^{\tL\tO}_k$ in $\bt$ take place at indices obtained from occurrences of $\tO\tL$ resp. $\tL\tO$, multiplied by $2^k$, these blocks occur alternatingly.
Assuming, in order to obtain a contradiction, that the gap sequence $(g_j)_{j\geq 0}$ for $w$ is automatic, we obtain a new automatic sequence $(g_{2j}+g_{2j+1})_{j\geq 0}$ as the sum of two automatic sequences (note that the characterisation involving the $2$-kernel~\eqref{eqn_k_kernel} immediately implies that $(g_{2j+\varepsilon})_{j\ge0}$, for $\varepsilon\in\{0,1\}$, is automatic).
By the alternating property, this is the gap sequence for $a^{\tO\tL}_k$, which is not automatic, as we have just seen. A contradiction!

\bigskip\noindent\textbf{Second case.}
Let $w$ be a factor of $a^{\tO\tO}_k$ or of $a^{\tL\tL}_k$.
This case is largely analogous.
We assume that $w$ be a factor of $a^{\tO\tO}_k$, but not of $a^{\tL\tL}_k$.
As in the case $a^{\tO\tL}_k$, the gap sequence for $w$ in this case is identical to the gap sequence for $a^{\tO\tO}_k$, and we only have to show that this sequence is not automatic.
We know already that the gap sequence for $\tO\tO$ is not automatic.
Therefore it suffices to prove that $\tau^k(\tO)\tau^k(\tO)$ can only appear at positions in $\bt$ divisible by $2^k$.
Suppose that we already know this for some $k\ge0$ (the case $k=0$ being again trivial).
Assume that
\begin{equation}\label{eqn_assumption2}
a^{\tO\tO}_{k+1}=\tau^k(\tO)\tau^k(\tL)\tau^k(\tO)\tau^k(\tL)
\mbox{ appears on some position }\ell.
\end{equation}
Since $\tau^k(\tO)\tau^k(\tL)=a^{\tO\tL}_k$, we know by hypothesis that $\ell\equiv 0\bmod 2^k$.
Assume that the case $\ell\equiv 2^k\bmod 2^{k+1}$ occurs.
We set $\ell=(2j+1)2^k$ for some $j\geq 0$.
Our assumption~\eqref{eqn_assumption2} implies
$\tau^k(\tL)=\tau^k(\bt_{2j+4})=\tau^k(\bt_{j+2})$ and therefore $\bt_{j+2}=\tL$, which implies that $\tau^{k+1}(\bt_{j+2})=\tau^k(\tL)\tau^k(\tO)$ appears on position $\ell+3\cdot 2^k=(2j+4)2^k$ in $\bt$.
On position $\ell$, we therefore see the factor
\[\tau^k(\tO)\tau^k(\tL)\tau^k(\tO)\tau^k(\tL)\tau^k(\tO),\]
which contradicts the overlap-freeness of $\bt$.

Again, the case that $w$ is a factor of $a^{\tL\tL}_k$, but not of $a^{\tO\tO}_k$, is analogous;
the case that it is a factor of both words can be handled as in the case $\{a^{\tO\tL}_k,a^{\tL\tO}_k\}$, this time with the help of the second chain of inequalities in~\eqref{eqn_chains}.

Summarising, we have shown the non-automaticity for all gap sequences for factors $w$ of $\bt$ of length $\geq 2$.

In order to finish the proof of Theorem~\ref{thm_main},
we still have to prove that the gap sequence is morphic for the `mixed cases'.
That is, assume that $w$ is a factor of two words of the form $a^{xy}_k$, where $x,y\in\{\tO,\tL\}$, and where $k$ is chosen minimal such that $w$ is a factor of at least one of $a^{\tO\tO}_k,a^{\tO\tL}_k,a^{\tL\tO}_k,a^{\tL\tL}_k$.
Let us begin with the case $\{a^{\tO\tL}_k,a^{\tL\tO}_k\}$.
The positions where $w$ appears in $\bt$ are given by $2^kj+\sigma_0$, where $\tb_j\tb_{j+1}=\tO\tL$, and $2^kj+\sigma_1$, where $\tb_j\tb_{j+1}=\tL\tO$. Here $\sigma_0,\sigma_1$ are the positions where the word $w$ appears in $a^{\tO\tL}_k$ and $a^{\tL\tO}_k$ respectively.
As before, this follows since $\bt_j\bt_{j+1}=\tO\tL$ is equivalent to $(\bt_\ell,\ldots,\bt_{\ell+2^{k+1}-1})=a^{\tO\tL}_k$, and the corresponding statement for $\tL\tO$. 
We see that it is sufficient to write $\bt$ as a concatenation of the words
\begin{equation}\label{eqn_w_choice_0110case}
w_{\ta}\coloneqq\tO\tL\tL,\quad w_{\tba}\coloneqq\tO\tL\tO,\quad w_{\tb}\coloneqq\tO\tL\tL\tO,\quad\mbox{and}\quad w_{\tc}=\tO\tL,
\end{equation}
since each word $w_x$ takes care of one $\tO\tL$-block, followed by one $\tL\tO$-block, and the gap sequence for $w$ is obtained by replacing each $w_x$ by a succession of two gaps. 
Applying the morphism $\tau$, we obtain
$\tau(w_{\ta})=w_{\ta}w_{\tba}$, $\tau(w_{\tba})=w_{\tb}w_{\tc}$, $\tau(w_{\tb})=w_{\ta}w_{\tba}w_{\tc}$, $\tau(w_{\tc})=w_{\tb}$.
This mimics the morphism $\psi$; proceeding as in the proof of Lemma~\ref{lem_correspondence} (alternatively, as in the proof of Lemma~\ref{lem_diffseq_substitution}), we obtain
\begin{equation}\label{eqn_w_concat}
\bt=w_{\bardiffseq_0}w_{\bardiffseq_1}w_{\bardiffseq_2}\cdots.
\end{equation}
Since $\bardiffseq$ is morphic, the succession of gaps with which $w$ occurs in $\bt$ is morphic by~\cite[Corollary~7.7.5]{AlloucheShallitBook} (that is, `morphic images of morphic sequences are morphic').

The case $\{a^{\tO\tO}_k,a^{\tL\tL}_k\}$ is similar.
Defining
\begin{equation*}
\tilde w_{\ta}\coloneqq\tO\tL\tL\tO\tL\tO,\quad
\tilde w_{\tba}\coloneqq\tO\tL\tL\tO\tO\tL,\quad
\tilde w_{\tb}\coloneqq\tO\tL\tL\tO\tL\tO\tO\tL,\quad\mbox{and}\quad
\tilde w_{\tc}\coloneqq\tO\tL\tL\tO,
\end{equation*}
it is straightforward to verify that
$\tau(\tilde w_{\ta})=\tilde w_{\ta}\tilde w_{\tba}$, $\tau(\tilde w_{\tba})=\tilde w_{\tb}\tilde w_{\tc}$, $\tau(\tilde w_{\tb})=\tilde w_{\ta}\tilde w_{\tba}\tilde w_{\tc}$, and $\tau(\tilde w_{\tc})=\tilde w_{\tb}$.
Again, we can spot the morphism $\psi$, and we obtain
\begin{equation}\label{eqn_w_concat_two}
\bt=\tilde w_{\bardiffseq_0}\tilde w_{\bardiffseq_1}\tilde w_{\bardiffseq_2}\cdots
\end{equation}
in exactly the same way as before.
Each of the words $w_x$ in this representation yields a block $\tL\tL$ in $\bt$, followed by a block $\tO\tO$.
Therefore, also in this case, the gap sequence for $w$ is a morphic image of a morphic sequence.
This finishes the proof of Theorem~\ref{thm_main}.
\qed

\begin{remark}
Let us have a closer look at the gaps in the `mixed case' $\{a^{\tO\tL}_k,a^{\tL\tO}_k\}$.
Let $\sigma_0$ be the index at which $w$ appears in $a^{\tO\tL}_k$, and $\sigma_1$ the index at which $w$ appears in $a^{\tL\tO}_k$.
By~\eqref{eqn_w_concat} and the choice~\eqref{eqn_w_choice_0110case}, each letter $x\in\{\ta,\tba,\tb,\tc\}$ in $\bardiffseq$ corresponds to two gaps, as follows.
\begin{equation}\label{eqn_bardiffseq_gap_correspondence_0110}
\begin{array}{l@{\hspace{2em}}l@{\hspace{2em}}l}
\mbox{Letter in $\bardiffseq$}&\mbox{Gap $1$}&\mbox{Gap $2$}\\
\ta&\sigma_1-\sigma_0+2^{k+1}&\sigma_0-\sigma_1+2^k\\
\tba&
\sigma_1-\sigma_0+2^k&\sigma_0-\sigma_1+2^{k+1}\\
\tb&
\sigma_1-\sigma_0+2^{k+1}&\sigma_0-\sigma_1+2^{k+1}\\
\tc&
\sigma_1-\sigma_0+2^k&\sigma_0-\sigma_1+2^k\\
\end{array}
\end{equation}
It follows that there are at most four gaps that can occur in this case.
For example, consider the gap sequence for the factor $w=\tO\tL\tO$.
In this case, $k=2$, and we have $a^{\tO\tL}_2=\tO\tL\tL\underline{\tO\tL\tO}\tO\tL$ and $a^{\tL\tO}_2=\tL\tO\underline{\tO\tL\tO}\tL\tL\tO$, where the occurrences of $w$ are underlined. We have $\sigma_0=3$ and $\sigma_1=2$.
This yields the gaps $3$, $5$, $7$, and $9$, occurring only in the combinations $(7,5)$, $(3,9)$, $(7,9)$, and $(3,5)$.
Noting also the first occurrence $\bt_3\bt_4\bt_5=\tO\tL\tO$,
the first few occurrences of $\tO\tL\tO$ in $\bt$ are at positions $3$, $10$, $15$, $18$, and $27$,
compare~\eqref{eqn_TM_32}.
In particular, the gap sequence is not of the form $2^{\ell}\diffseq$ or $2^{\ell}\check\diffseq$ for some $\ell\geq 0$, each of which has only three different values.

Similar considerations hold for the case $\{a^{\tO\tO}_k,a^{\tL\tL}_k\}$.
More precisely,
let $\sigma_0$ be the index at which $w$ appears in $a^{\tL\tL}_k$ and $\sigma_1$ the index at which $w$ appears in $a^{\tO\tO}_k$.
Each letter occurring in $\bardiffseq$ corresponds to two gaps for $w$, as follows.
\begin{equation}\label{eqn_bardiffseq_gap_correspondence_0011}
\begin{array}{l@{\hspace{2em}}l@{\hspace{2em}}l}
\mbox{Letter in $\bardiffseq$}&\mbox{Gap $1$}&\mbox{Gap $2$}\\
\ta&\sigma_1-\sigma_0+4\cdot 2^k&\sigma_0-\sigma_1+2\cdot 2^k\\
\tba&
\sigma_1-\sigma_0+2\cdot 2^k&\sigma_0-\sigma_1+4\cdot 2^k\\
\tb&
\sigma_1-\sigma_0+4\cdot 2^k&\sigma_0-\sigma_1+4\cdot 2^k\\
\tc&
\sigma_1-\sigma_0+2\cdot 2^k&\sigma_0-\sigma_1+2\cdot 2^k.
\end{array}
\end{equation}
An example for this case is given by the word $\tO\tO\tL\tL\tO$, which is a factor of $a^{\tO\tO}_2=\tO\tL\tL\tO\tO\tL\tL\tO$ and of $a^{\tL\tL}_2=\tL\tO\tO\tL\tL\tO\tO\tL$.
We have $\sigma_0=1$ and $\sigma_1=3$, and therefore the gaps $6$, $10$, $14$, and $18$, which appear as pairs $(18,6)$, $(10,14)$, $(18,14)$, and $(10,6)$.
\end{remark}

\section{The structure of the sequence $\mathbf A$}\label{sec_explore}
In this section, we investigate the infinite word $\autoseq$,
in particular by extending it to a word over a $7$-letter alphabet.
This extension allows us to better understand the structure of $\autoseq$,
and gives us a tool to handle the discrepancy $D_N$.
In particular, we prove Theorem~\ref{thm_discrepancy}.
\subsection{$\mathbf A$ is automatic}
It has been known since Berstel~\cite{B1979} that $\autoseq$ is $2$-automatic.
In this section, we re-prove this statement using slightly different notation.
Note that we had similar proofs (of Lemmas~\ref{lem_correspondence} and~\ref{lem_diffseq_substitution}) in the first part of this paper.
First of all, we recapture Berstel's $2$-uniform morphism.
Introducing an auxiliary letter $\tbb$, we have the morphism
$\overline\varphi$ as well as the coding $\pi$:
\begin{equation}\label{eqn_berstel}
\begin{array}{lllll}
\overline\varphi:&
\ta\mapsto \ta\tb,&
\tb\mapsto\tc\ta,&
\tbb\mapsto\ta\tc,&
\tc\mapsto\tc\tbb,
\\
\pi:&\ta\mapsto\ta,&\tb\mapsto\tb,&\tbb\mapsto\tb,&\tc\mapsto\tc.
\end{array}
\end{equation}
We wish to prove that
\begin{equation}\label{eqn_berstel_identity}
\pi(\barautoseq)=\autoseq,
\end{equation}
where
$\barautoseq$ is the fixed point of $\overline\varphi$ starting with $\ta$.
For this, we will show, by induction on $k\geq 0$,
that the initial segment
\[s_k\coloneqq\overline\varphi^k(\ta\tb\tc)\]
of $\barautoseq$, of length $3\cdot 2^k$,
is a concatenation of the three words
$w_0=\ta\tb\tc$, $w_1=\ta\tc$, and $w_2=\tbb$.
We also call the words $w_j$ `base words' in this context, and the latter statement `concatenation property'.
Having proved this property, we use (recall the morphism $\varphi$ defined in~\eqref{eqn_varphi_morphism})
\begin{align}\label{eqn_basewords_1}
\begin{array}{lll}
\pi\bigl(\overline\varphi(w_1)\bigr)
&=\ta\tb\tc\ta\tc\tb&=\varphi(\pi(w_1));\\
\pi\bigl(\overline\varphi(w_2)\bigr)
&=\ta\tb\tc\tb&=\varphi(\pi(w_1));\\
\pi\bigl(\overline\varphi(w_3)\bigr)
&=\ta\tc&=\varphi(\pi(w_3)),
\end{array}
\end{align}
in order to obtain
\begin{equation}\label{eqn_commutative}
\varphi\bigl(\pi(s_k)\bigr)
=\pi\bigl(\overline\varphi(s_k)\bigr)
\end{equation}
for all $k\geq 0$, by concatenation.
In other words, $\varphi$ and $\overline{\varphi}$ act in the same way on an initial segment of $\barautoseq$ of length $3\cdot 2^k$.

We may also display the relation~\eqref{eqn_commutative} graphically.
Define $S=\{s_k:k\geq 0\}\subseteq \{\ta,\tb,\tbb,\tc\}^{\mathbb N}$ and $\Omega=\{\ta,\tb,\tc\}^{\mathbb N}$.
Then the following diagram is commutative.
\begin{center}
\begin{tikzpicture}[xscale=2,yscale=1.5]
\node (la)	at (0,1)	{$S$};
\node (lb)	at (0,0)	{$S$};
\node (ra)	at (1,1)	{$\Omega$};
\node (rb)	at (1,0)	{$\Omega$};
\path[>=latex,->,font=\scriptsize]
(la)    edge node [auto]   {$\pi$}                 (ra)
(ra)    edge node [auto]   {$\varphi$}           (rb)
(la)    edge node [auto,swap]   {$\overline{\varphi}$}           (lb)
(lb)    edge node [auto,swap]   {$\pi$}                 (rb);
\end{tikzpicture}
\end{center}
Gluing together copies of this diagram, we obtain, for all $\ell\geq 1$,
\begin{align*}
\varphi^\ell\bigl(s_0\bigr)&=\varphi^\ell\bigl(\pi(s_0)\bigr)
=
\varphi^{\ell-1}\bigl(\pi(\overline\varphi(s_0))\bigr)
\\&=
\varphi^{\ell-2}\bigl(\varphi(\pi(\overline\varphi(s_0)))\bigr)
=
\varphi^{\ell-2}\bigl(\pi(\overline\varphi^2(s_0))\bigr)
=\cdots=\pi\bigl(\overline\varphi^\ell(s_0)\bigr).
\end{align*}
For each index $j\geq 0$, choose $\ell$ so large that $3\cdot 2^\ell\geq j$.
Then
\[\autoseq_i=\varphi^\ell(s_0)\bigr\rvert_i=\pi\bigl(\overline\varphi^\ell(s_0)\bigr)\bigr\rvert_i=
\pi\bigl(\overline\varphi^\ell(s_0)\bigr\rvert_i\bigr)=
\pi(\barautoseq_i)\]
for $0\leq i<j$.
Therefore the infinite word $\autoseq$ is $2$-automatic, being the coding under $\pi$ of the $2$-automatic sequence $\barautoseq$, and thus we have derived~\eqref{eqn_berstel_identity} from the concatenation property.

We still have to prove that $s_k$ is a concatenation of the base words.
Clearly, this holds for $s_0=\ta\tb\tc=w_0$.
Assume that we have already established that
$s_k=w_{\varepsilon_0}w_{\varepsilon_1}\cdots$
for some $\varepsilon_j\in\{0,1,2\}$.
We have
\begin{align*}
\overline\varphi(w_1)&=\ta\tb\tc\ta\tc\tbb=w_0w_1w_2,\\
\overline\varphi(w_2)&=\ta\tb\tc\tbb=w_0w_2,\\
\overline\varphi(w_3)&=\ta\tc=w_1,
\end{align*}
and thus
\[s_{k+1}=\overline\varphi(s_k)
=\overline\varphi(w_{\varepsilon_0})\overline\varphi(w_{\varepsilon_1})\cdots\]
is a concatenation of the $w_j$ too.
This proves~\eqref{eqn_berstel_identity}.

Complementing this result, we note that Berstel~\cite[Corollaire~7]{B1979} also proved that $\autoseq$ itself is not a fixed point of (the extension of) a uniform morphism.

\subsection{Transforming $\mathbf A$}
We will identify \emph{circular shifts}, or \emph{rotations}, of factors of length $L\geq 2$ appearing in the sequence $\autoseq$. Such a rotation of a word $(a_i)_{i\geq 0}$ replaces the subword $a_j\,a_{j+1}\cdots a_{j+L-2}\,a_{j+L-1}$ by $a_{j+1}\cdots a_{j+L-2}\,a_{j+L-1}\,a_j$ (rotation to the left), or $a_{j+L-1}\,a_j\,a_{j+1}\cdots a_{j+L-2}$ (rotation to the right), respectively.

Carrying out a certain number of such rotations, we will see that the sequence $\mathbf A$ is reduced to the periodic word $(\ta\tb\tc)^\omega$.
Of course, this is possible for any word containing an infinite number of each of $\ta$, $\tb$, and $\tc$, and it can be achieved in uncountably many ways.
In our case however, an admissible sequence of rotations can be made very explicit, by defining a new morphism $\varphi^+$.
This morphism has the fixed point $\barautoseq$, which maps to $\autoseq$ under a coding.
From this augmented sequence, we will see very clearly the `nested structure' of the above-mentioned rotations.
In particular, we can find a certain \emph{non-crossing matching}, defined in~\eqref{eqn_def_ncm}, describing the intervals that we perform rotations on, and the direction of each rotation.
Moreover, in the process we learn something about the discrepancy of $\tO\tL$-blocks in $\bt$, which was defined in~\eqref{eqn_def_discrepancy}.
Let us consider the iteration $\overline\varphi^2$ of Berstel's morphism:
\begin{equation}\label{eqn_berstelsquare}
\begin{array}{lllll}
\overline\varphi^2:&
\ta\mapsto \ta\tb\tc\ta,&
\tb\mapsto\tc\tbb\ta\tb,&
\tbb\mapsto\ta\tb\tc\tbb,&
\tc\mapsto\tc\tbb\ta\tc.
\end{array}
\end{equation}
We introduce certain decorations--- \emph{connectors} --- of the letters. Their meaning will become clear in a moment.
Based on the morphism $\overline\varphi^2$, we define the following decorated version, which is a morphism on the $7$-letter alphabet
\begin{equation}\label{eqn_K_def}
K=\{\ta,\Lbbleft,\Lbbright,\Lbleft,\Lbright,\Lcleft,\Lcright\}.
\end{equation}

\begin{equation}\label{eqn_decorated_substitution}
\varphi^+:
\begin{array}{r@{\hspace{0.3em}}l@{\hspace{1.5em}}r@{\hspace{0.3em}}l@{\hspace{1.5em}}r@{\hspace{0.3em}}l@{\hspace{1.5em}}r@{\hspace{0.3em}}l}
\ta&\mapsto
\begin{tikzpicture}[xscale=0.2,yscale=1,anchor=base, baseline]
\node[inner sep=1pt] (a0)	at (0,0)	{$\ta$};
\node[inner sep=1pt] (a1)	at (1,0)	{$\tb$};
\node[inner sep=1pt] (a2)	at (2,0)	{$\tc$};
\node[inner sep=1pt] (a3)	at (3,0)	{$\ta$};
\node[inner sep=1pt] (a4)	at (3.7,0)	{,};
\draw (a1.south)  -- ++(0,-0.1) --  ++(\clength,0);
\draw (a2.south)  -- ++(0,-0.1) --  ++(-\clength,0);
\end{tikzpicture}
&
\begin{tikzpicture}[xscale=0.2,yscale=1,anchor=base, baseline]
\node[inner sep=1pt] (a1)	at (0,0)	{$\tbb$};
\draw (a1.south)  -- ++(0,-0.1) --  ++(-\clength,0);
\end{tikzpicture}
&\mapsto
\begin{tikzpicture}[xscale=0.2,yscale=1,anchor=base, baseline]
\node[inner sep=1pt] (a0)	at (0,0)	{$\ta$};
\node[inner sep=1pt] (a1)	at (1,0)	{$\tb$};
\node[inner sep=1pt] (a2)	at (2,0)	{$\tc$};
\node[inner sep=1pt] (a3)	at (3,0)	{$\tbb$};
\node[inner sep=1pt] (a4)	at (3.7,0)	{,};
\draw (a1.south)  -- ++(0,-0.1) --  ++(-\clength,0);
\draw (a2.south)  -- ++(0,-0.1) --  ++(\clength,0);
\draw (a3.south)  -- ++(0,-0.1) --  ++(-\clength,0);
\end{tikzpicture}
&
\begin{tikzpicture}[xscale=0.2,yscale=1,anchor=base, baseline]
\node[inner sep=1pt] (a1)	at (0,0)	{$\tbb$};
\draw (a1.south)  -- ++(0,-0.1) --  ++(\clength,0);
\end{tikzpicture}
&\mapsto
\begin{tikzpicture}[xscale=0.2,yscale=1,anchor=base, baseline]
\node[inner sep=1pt] (a0)	at (0,0)	{$\ta$};
\node[inner sep=1pt] (a1)	at (1,0)	{$\tb$};
\node[inner sep=1pt] (a2)	at (2,0)	{$\tc$};
\node[inner sep=1pt] (a3)	at (3,0)	{$\tbb$};
\node[inner sep=1pt] (a4)	at (3.7,0)	{,};
\draw (a1.south)  -- ++(0,-0.1) --  ++(\clength,0);
\draw (a2.south)  -- ++(0,-0.1) --  ++(-\clength,0);
\draw (a3.south)  -- ++(0,-0.1) --  ++(\clength,0);
\end{tikzpicture}
\\
\begin{tikzpicture}[xscale=0.2,yscale=1,anchor=base, baseline]
\node[inner sep=1pt] (a1)	at (0,0)	{$\tb$};
\draw (a1.south)  -- ++(0,-0.1) --  ++(-\clength,0);
\end{tikzpicture}
&\mapsto
\begin{tikzpicture}[xscale=0.2,yscale=1,anchor=base, baseline]
\node[inner sep=1pt] (a0)	at (0,0)	{$\tc$};
\node[inner sep=1pt] (a1)	at (1,0)	{$\tbb$};
\node[inner sep=1pt] (a2)	at (2,0)	{$\ta$};
\node[inner sep=1pt] (a3)	at (3,0)	{$\tb$};
\node[inner sep=1pt] (a4)	at (3.7,0)	{,};
\draw (a0.south)  -- ++(0,-0.1) --  ++(\clength,0);
\draw (a1.south)  -- ++(0,-0.1) --  ++(-\clength,0);
\draw (a3.south)  -- ++(0,-0.1) --  ++(-\clength,0);
\end{tikzpicture}
&
\begin{tikzpicture}[xscale=0.2,yscale=1,anchor=base, baseline]
\node[inner sep=1pt] (a1)	at (0,0)	{$\tb$};
\draw (a1.south)  -- ++(0,-0.1) --  ++(\clength,0);
\end{tikzpicture}
&\mapsto
\begin{tikzpicture}[xscale=0.2,yscale=1,anchor=base, baseline]
\node[inner sep=1pt] (a0)	at (0,0)	{$\tc$};
\node[inner sep=1pt] (a1)	at (1,0)	{$\tbb$};
\node[inner sep=1pt] (a2)	at (2,0)	{$\ta$};
\node[inner sep=1pt] (a3)	at (3,0)	{$\tb$};
\node[inner sep=1pt] (a4)	at (3.7,0)	{,};
\draw (a0.south)  -- ++(0,-0.1) --  ++(\clength,0);
\draw (a1.south)  -- ++(0,-0.1) --  ++(-\clength,0);
\draw (a3.south)  -- ++(0,-0.1) --  ++(\clength,0);
\end{tikzpicture}
&
\begin{tikzpicture}[xscale=0.2,yscale=1,anchor=base, baseline]
\node[inner sep=1pt] (a1)	at (0,0)	{$\tc$};
\draw (a1.south)  -- ++(0,-0.1) --  ++(-\clength,0);
\end{tikzpicture}
&\mapsto
\begin{tikzpicture}[xscale=0.2,yscale=1,anchor=base, baseline]
\node[inner sep=1pt] (a0)	at (0,0)	{$\tc$};
\node[inner sep=1pt] (a1)	at (1,0)	{$\tbb$};
\node[inner sep=1pt] (a2)	at (2,0)	{$\ta$};
\node[inner sep=1pt] (a3)	at (3,0)	{$\tc$};
\node[inner sep=1pt] (a4)	at (3.7,0)	{,};
\draw (a0.south)  -- ++(0,-0.1) --  ++(-\clength,0);
\draw (a1.south)  -- ++(0,-0.1) --  ++(\clength,0);
\draw (a3.south)  -- ++(0,-0.1) --  ++(-\clength,0);
\end{tikzpicture}
&
\begin{tikzpicture}[xscale=0.2,yscale=1,anchor=base, baseline]
\node[inner sep=1pt] (a1)	at (0,0)	{$\tc$};
\draw (a1.south)  -- ++(0,-0.1) --  ++(\clength,0);
\end{tikzpicture}
&\mapsto
\begin{tikzpicture}[xscale=0.2,yscale=1,anchor=base, baseline]
\node[inner sep=1pt] (a0)	at (0,0)	{$\tc$};
\node[inner sep=1pt] (a1)	at (1,0)	{$\tbb$};
\node[inner sep=1pt] (a2)	at (2,0)	{$\ta$};
\node[inner sep=1pt] (a3)	at (3,0)	{$\tc$};
\node[inner sep=1pt] (a4)	at (3.7,0)	{.};
\draw (a0.south)  -- ++(0,-0.1) --  ++(\clength,0);
\draw (a1.south)  -- ++(0,-0.1) --  ++(-\clength,0);
\draw (a3.south)  -- ++(0,-0.1) --  ++(\clength,0);
\end{tikzpicture}
\end{array}
\end{equation}
This morphism has a unique fixed point $\autoseqplus$ starting with $\ta$.
The image of $\autoseqplus$ under the obvious coding $\gamma$ given by
\begin{equation}\label{eqn_big_coding}
\gamma:
\begin{array}{lllll}
&\ta\mapsto \ta,\\
&\Lbleft\mapsto\tb,&
\Lbright\mapsto\tb,&
\Lbbleft\mapsto\tb,&\Lbbright\mapsto\tb,\\
&\Lcleft\mapsto\tc,&\Lcright\mapsto\tc
\end{array}
\end{equation}
yields the sequence $\autoseq$.
Based on this, we will speak of \emph{letters of types} $\ta$, $\tb$, and $\tc$, thus referring to letters from $\{\ta\}$, $\{\Lbbleft,\Lbbright,\Lbleft,\Lbright\}$, and $\{\Lcleft,\Lcright\}$, respectively.

From the substitution~\eqref{eqn_decorated_substitution}, we can immediately derive the following lemma.
\begin{lemma}\label{lem_local}
Let $j\geq 1$, and $(x,y,z)=(\autoseqplus_{j-1},\autoseqplus_j,\autoseqplus_{j+1})$. Then
\begin{equation}\label{eqn_neighbourhood}
\begin{array}{rcl@{\hspace{3em}}rcl}
y=\Lbbleft&\Rightarrow& xyz=\Lcright\!\Lbbleft\ta;&
y=\Lbbright&\Rightarrow& xyz=\Lcleft\!\Lbbright\ta;\\
y=\Lbleft&\Rightarrow& xyz=\ta\Lbleft\!\Lcright;&
y=\Lbright&\Rightarrow& xyz=\ta\Lbright\!\Lcleft.
\end{array}
\end{equation}
\end{lemma}

We wish to connect the `loose ends' of the connectors ---  we say that two connectors at indices $i<j$ \emph{match} if the connector at $i$ points to the right and the connector at $j$ points to the left.
The very simple algorithm \texttt{FindMatching} joins matching connectors, beginning with shortest connections.
Only pairs of \emph{free} connectors are connected, that is, each letter may be the starting point of only one link.

\vspace{1em}

\noindent\texttt{procedure FindMatching(w):}

\noindent\hspace{1em}\texttt{M$\leftarrow$\string{\string};}

\noindent\hspace{1em}\texttt{SelectedIndices$\leftarrow$\string{\string};}

\noindent\hspace{1em}\texttt{n$\leftarrow$1;}

\noindent\hspace{1em}\texttt{while n < w.length:}

\smallskip\noindent\hspace{2em}\texttt{for all i }such that there are matching connectors at \texttt{i} and \texttt{i+n}:

\smallskip\noindent\hspace{3em}\texttt{if i} $\notin$ \texttt{SelectedIndices and i+n} $\notin$ \texttt{SelectedIndices:}

\smallskip\noindent\hspace{4em}Add the pair \texttt{(i,i+n)} to the set \texttt{M;}

\smallskip\noindent\hspace{4em}Add \texttt{i} and \texttt{i+n} to the set \texttt{SelectedIndices;}

\medskip\noindent\hspace{2em}\texttt{n$\leftarrow$n+1;}

\noindent\hspace{1em}\texttt{return M;}

\noindent\texttt{end.}
\begin{center}Algorithm \texttt{FindMatching}. Link free connectors
\end{center}
\vspace{3em}

Note that we have to pay attention that previously selected indices are not chosen again, whence the introduction of \texttt{SelectedIndices}.
A connection between the two letters at indices $i$ and $j$ is just a different name for the pair $(i,j)$.
For any finite word $w$ over the alphabet $K$
this procedure yields a (possibly empty) set $M(w)$ of pairs $(i,j)$ of indices.

We wish to prove that the algorithm is \emph{monotone}.
\begin{lemma}
Let $w$ and $w'$ be finite words over the alphabet $K$, and assume that
$w$ is an initial segment of $w'$.
Let $M(w)$ resp. $M(w')$ be the sets of pairs found by the algorithm
\textup{\texttt{FindMatching}}. Then
\begin{equation}\label{eqn_matching_inclusion}
M(w)\subseteq M(w').
\end{equation}
\end{lemma}
\begin{proof}
We show this by induction on the length $j$ of $w$. Clearly this holds for $j=0$. Let us append a symbol $x\in K$ to $w$ (at position $j$).
Define $M_\ell(w)$ as the set of connections $(a,b)$ for $w$ of length strictly smaller than $\ell$, found by the algorithm. Define $M_\ell(wx)$ analogously.
We prove by induction on $\ell$ that $M_\ell(w)\subseteq M_\ell(wx)$, and that,
if the inclusion is strict, we have
$M_\ell(wx)=M_\ell(w)\cup\{(i,j)\}$ for some $i<j$.
Suppose that this is true for some $\ell$ (clearly it holds for $\ell=0$). We distinguish between two cases.
{\inparaenum[(i)] \item
If $(i,j)\not\in M_\ell(wx)$ for all $i$, we have $M_\ell(w)=M_\ell(wx)$ by hypothesis;
we add each pair $(a,b)$ with $b<j$ having matching connectors and such that $b-a=\ell$ to the sets $M_\ell(w)$ and $M_\ell(wx)$, and possibly one more pair $(i,j)$, for some $i<j$, to $M_\ell(wx)$.
 \item If $(i,j)\in M_\ell(wx)$ for some $i$, we have $\ell>j-i$ by the definition of $M_\ell(wx)$;
we add the pairs $(a,b)$, with $b<j$, having matching connectors and such that $a\neq i$ and $b-a=\ell$ to both sets $M_\ell(w)$ and $M_\ell(wx)$.
There are clearly no more pairs added to $M_\ell(wx)$, since $i$ and $j$ are already taken; moreover, the condition that $\ell>j-i$ renders impossible the chance of another connection $(i,b)$, where $b<j$, to be added to $M_\ell(w)$.
}
\end{proof}

We extend $M$ to a function on all (finite or infinite) words $w$ over $K$, in the following obvious way:
for each $\ell$, form the set $\tilde M_\ell(w)$ of all pairs $(a,b)$ satisfying $b-a=\ell$, having matching connectors, such that neither $a$ nor $b$ is a component of any $\tilde M_{\ell'}(w)$, where $\ell'<\ell$.
Set $\tilde M(w)=\bigcup_{\ell\geq 1}\tilde M_\ell(w)$.
The following lemma gives us a method to compute a matching for an infinite word by only looking at finite segments.
\begin{lemma}\label{lem_commutative}
Let $w$ be an infinite word over $K$.
Then
\begin{equation}\label{eqn_limit}
\bigcup_{j\geq 0}M\bigl(w\bigr\rvert_{[0,j)}\bigr)=\tilde M(w).
\end{equation}
\end{lemma}
\begin{proof}
Let $M_\ell(w)$ be the set of connections added in step $\ell$ of the algorithm \texttt{FindMatching}.
We prove, more generally, that
\begin{equation}\label{eqn_more_generally}
\bigcup_{j\geq 0}M_\ell\bigl(w\bigr\rvert_{[0,j)}\bigr)=\tilde M_\ell(w).
\end{equation}
We prove this by induction on $\ell$, and we start at connections of length $\ell=1$.
Let $(i,i+1)\in M_\ell\bigl(w\bigr\rvert_{[0,j)}\bigr)$.
Then there is a pair of matching connectors at indices $i$ and $i+1$ (where $i+1<j$),
and therefore this pair is also contained in $\tilde M_1(w)$.
This proves the inclusion ``$\subseteq$''.
On the other hand, if $(i,i+1)$ is a link connecting matching connectors in $w$,
this link is also to be found in the sequence $w\bigr\rvert_{[0,i+2)}$, hence the inclusion ``$\supseteq$''.
Assume that~\eqref{eqn_more_generally} holds for some $\ell\geq 1$.
If the algorithm finds a pair $(i,i+\ell)$ of matching connectors in $w\bigr\rvert_{[0,j)}$, where $(i,i+\ell)\not\in M_\ell\bigl(w\bigr\rvert_{[0,j)}\bigr)$, this pair trivially also matches in the (unrestricted) word $w$. By hypothesis, the connectors at $i$ and $i+\ell$ are not used by $\tilde M_\ell(w)$, hence the inclusion ``$\subseteq$''.
On the other hand, a link $(i,i+\ell)$ of matching connectors in $w$ that is still free in step $\ell$ is also free in $w\bigr\rvert_{[0,i+\ell+1)}$ by hypothesis, which proves~\eqref{eqn_more_generally} and hence the lemma.
\end{proof}

Our algorithm avoids crossing connections: if $i<j<k<\ell$ were indices such that $(i,k)\in M(w)$ and $(j,\ell)\in M(w)$, then the connector at index $j$ is pointing to the right, and the one at $k$ to the left, so the shorter connection $(j,k)$ would have been chosen earlier.
This contradicts the construction rule that indices may only be chosen once.

More generally, a \emph{non-crossing matching} for a word $w$ over $K$ (finite or infinite) is a set $M$ of pairs $(i,j)$ such that
\begin{align}
&i<j&&\mbox{for all }(i,j)\in M,\nonumber\\
&w_iw_j\in\{\Lbright\!\Lcleft,\Lbbright\!\Lcleft,\Lcright\!\Lbleft,\Lcright\!\Lbbleft\}&&\mbox{for all }(i,j)\in M,\nonumber\\
&w_i=a&&\mbox{for all }i\not\in\textstyle\bigcup M,\label{eqn_def_ncm}\\
&\left.\!\!\begin{array}{@{}ll}(i,j)=(k,\ell)&\mbox{\textsf{or} }\\
i<k<\ell<j&\mbox{\textsf{or}}\\
k<i<j<\ell
\end{array}\right\}
&&\mbox{for all }(i,j)\in M,\ (k,\ell)\in M\nonumber.
\end{align}
Here $\bigcup M=\{i:(i,j)\in M\mbox{ for some }j\mbox{ \textsf{or} }(j,i)\in M\mbox{ for some }j\}$.

We call a word $w$ \emph{closed} if there exists a non-crossing matching for $w$.

\begin{lemma}\label{lem_FindMatching_findsit}
Let $w$ be a word over $K$. There is at most one non-crossing matching for $w$.
If there exists one, \textup{\texttt{FindMatching}} generates it by virtue of~\eqref{eqn_limit}.
\end{lemma}
\begin{proof}
Let $m$ be a non-crossing matching of $w$.
Since all connectors have to connect to something and the connecting lines must not cross, we see that all pairs $(i,i+1)$ of indices where matching connectors appear have to be contained in $m$. It follows that $M_1\bigl(w\bigr\rvert_{[0,j)}\bigr)\subseteq m$ for all $j$, and therefore $\tilde M_1(w)\subseteq m$ by~\eqref{eqn_more_generally}.
On the other hand, the definition of a non-crossing matching only allows matching connectors, therefore each connection $(i,i+1)$ in $m$ is found by \texttt{FindMatching}, for $j=i+2$.

Similar reasoning applies for longer connections too. 
Let us assume that the set of connections of length $<\ell$ coming from \texttt{FindMatching} is the same as the set of connections of length $<\ell$ contained in $m$.
Assume that $i$ is an index such that the connectors at indices $i$ and $i+\ell$ match, and neither $i$ nor $i+\ell$ appears in a connection of length $<\ell$ in $m$.
Since $m$ is a matching, the connector at index $i$ has to be linked to a connector at an index $j>i$. Indices $j\in\{i+1,\ldots,i+\ell-1\}$ are excluded by our hypothesis, indices $j>i+\ell$ are impossible by the non-crossing property, therefore $(i,i+\ell)\in m$.
Again, other connections of length $\ell$ cannot appear in $m$, therefore \texttt{FindMatching} finds all pairs $(i,i+\ell)$ contained in $m$.
This completes our argument by induction.
Therefore $m=\tilde M(w)$, and both statements of Lemma~\ref{lem_FindMatching_findsit} follow.
\end{proof}

\begin{lemma}\label{lem_autoseq_closed}
The sequence $\autoseqplus$ is closed.
\end{lemma}
\begin{proof}
First of all we note that it is sufficient to prove that $\varphi^+$ maps closed words $w$ to closed words.
If this is established, we obtain, by induction, that the initial segments $(\varphi^+)^k(\ta)$ of $\autoseqplus$ are closed.
Since non-crossing matchings are unique, the corresponding sequence $(m_k)_{k\geq 0}$ of non-crossing matchings satisfies $m_k\subseteq m_{k+1}$, and $\bigcup_{k\geq 0}m_k$ is easily seen to be the desired matching for $\autoseqplus$.

We prove by induction on the length $n$ of a closed word $w$ that $\varphi^+(w)$ is closed.
This is obvious for the closed words of length $n\le2$: the word $\varphi^+(\ta)=\ta\Lbright\!\Lcleft\ta$ is closed, and the cases $\Lbright\Lcleft$, $\Lbbright\Lcleft$, $\Lcright\Lbleft$, and $\Lcright\Lbbleft$ are also easy.
Moreover, a concatenation of two closed words is also closed: one of the matchings has to be shifted (both components of each entry have to be shifted), and we only have to form the union of the matchings.

If $w$ is of the form $\Lbbright C\hspace{-1pt}\Lcleft$ for some nonempty word $C$ over $K$,
we obtain a non-crossing matching for $C$ by stripping the pair $(1,n)$ from a corresponding matching for $w$. Therefore $C$ is closed.
Applying $\varphi^+$, we see that
\begin{equation}\label{eqn_preservation}
\varphi^+(w)=\ta\Lbright\!\Lcleft\!\Lbbright\varphi^+(C)\Lcleft\!\Lbbright\hspace{-0.5pt}\ta\hspace{-1pt}\Lcleft.\end{equation}
This is closed by our hypothesis, since $C$ is shorter than $w$.
The other case $\Lcright C\hspace{-0.0pt}\Lbleft$ is analogous (note that there are no more cases by~\eqref{eqn_neighbourhood}), and the proof is complete.
\end{proof}
\begin{remark}
We note that this proof can also be used to show that 
the substitution $\varphi^+$ \emph{respects} non-crossing matchings, in the following sense.
If $m$ is a non-crossing matching for $w$, then there exists a (unique) non-crossing matching $m'$ for $\varphi^+(w)$;
the matching $m$ can be recovered from $m'$ by omitting certain links, and applying a renaming $(i,j)\mapsto (\mu(i),\mu(j))$ to the remaining links, where $\mu:\mathbb N\rightarrow\mathbb N$ is nonincreasing.
The proof is not difficult: if this procedure works for the closed word $C$,
we can also carry this out for $\Lbbright C\hspace{-1pt}\Lcleft$ by~\eqref{eqn_preservation}; we see that the additional link $\Lbbright \cdots \Lcleft$ is still present in
$\varphi^+\bigl(\Lbbright C\hspace{-1pt}\Lcleft\bigr)$.
Also, the procedure of recovering $m'$ from $m$ is compatible with concatenations of closed words $C$ and $D$, as a matching for $\varphi^+(CD)=\varphi^+(C)\,\varphi^+(D)$ does not connect letters in $\varphi^+(C)$ and $\varphi^+(D)$.
\end{remark}

The construction of the matching in the proof of Lemma~\ref{lem_autoseq_closed} also shows the following result.
\begin{corollary}
Let $m$ be the non-crossing matching for $\autoseqplus$.
By virtue of $m$, each letter of type $\tc$ is connected to exactly one letter, which is of type $\tb$, and each letter of type $\tb$ is connected to exactly one letter, which is of type $\tc$.
\end{corollary}

Our interest in the link structure of $\autoseqplus$ stems from the fact that we may transform the sequence $\autoseq$ into a periodic one, using the following transparent mechanism.
Let $m$ be the non-crossing matching for $\autoseqplus$,
 and let $((i_k,j_k))_{k\geq 0}$ be an enumeration of $m$ such that $(j_k-i_k)_{k\geq 0}$ is nondecreasing.
We define a sequence $(A^{(k)})_{k\geq 0}$ as follows.
\begin{itemize}
\item
Set $A^{(0)}=\autoseqplus$.
\item
Let $k\geq 0$.
If $A^{(k)}_{i_k}=\Lcright$, we \emph{rotate} the letters in $A^{(k)}$ with indices $i_k,i_k+1,\ldots,j_k$ \emph{to the right} by one place, yielding $A^{(k+1)}$.
Otherwise, we necessarily have $A^{(k)}_{j_k}=\Lcleft$ and we rotate the letters with indices $i_k,i_k+1,\ldots,j_k-1$ \emph{to the left} by one place.
\end{itemize}
In more colourful language, in each step some letter of type $\tb$ is moved along its connecting link and inserted just before the letter of type $\tc$ it is connected to.
Note that due to the monotonicity requirement and the non-crossing property, the $k$th rotation does not change the indices at which the subsequent rotations are carried out.
Therefore the sequence $(A^{(k)})_{k\ge0}$ is well-defined.
Moreover, the result does not depend on the particular nondecreasing enumeration of $m$ for the same reasons.
Since the first $N$ indices eventually remain unchanged, the limit
\begin{equation}\label{eqn_A_def}
\rho(\autoseqplus)\coloneqq\gamma\bigl(\lim_{k\rightarrow\infty}A^{(k)}\bigr)
\end{equation}
exists (note that $\gamma$, defined in~\eqref{eqn_big_coding}, replaces each letter of type $x$ by $x$).
The definition of $A^{(k)}$ is summarised in the algorithm
\texttt{RotateAlongLinks}.
As in the case of the algorithm \texttt{FindMatching}, we require a finite word $w$ (and a finite set $m\subset \mathbb N^2$) as input in order to guarantee finite running time.
\vspace{1em}

\noindent\texttt{procedure RotateAlongLinks(w,m)}

\noindent\hspace{1em}\texttt{if m} is not a non-crossing matching for \texttt{w:}

\smallskip\noindent\hspace{2em}
\texttt{exit(\textrm{Error: a non-crossing matching is required});}

\noindent\hspace{1em}Create a list \texttt{m'} from \texttt{m}, ordered by \texttt{SecondComponent-FirstComponent}

\noindent\hspace{1em}\texttt{for p in m':}

\smallskip\noindent\hspace{2em}
\texttt{i$\leftarrow$p.FirstComponent;}

\noindent\hspace{2em}
\texttt{j$\leftarrow$p.SecondComponent;}

\noindent\hspace{2em}
\texttt{if w[i]=}$\Lcright$\texttt{:}

\noindent\hspace{3em}
\texttt{\#Rotate right}

\noindent\hspace{3em}
\texttt{(w[i],...,w[j-1],w[j])}$\leftarrow$\texttt{(w[j],w[i],...,w[j-1]);}

\noindent\hspace{2em}
\texttt{else:}

\noindent\hspace{3em}
\texttt{\#In this case, w[j]=}$\Lcleft$\texttt{.\ Rotate left}

\noindent\hspace{3em}
\texttt{(w[i],w[i+1],...,w[j-1])}$\leftarrow$\texttt{(w[i+1],...,w[j-1],w[i]);}

\noindent\hspace{1em}\texttt{return w;}

\noindent\texttt{end.}
\begin{center}Algorithm \texttt{RotateAlongLinks}. Transform a closed word according to a non-crossing matching
\end{center}
\vspace{1em}
By the above remarks,
the words \texttt{RotateAlongLinks}$\bigl(w_k,M(w_k)\bigr)$ converge
to $\rho(\autoseqplus)$ as $k\rightarrow\infty$, where $w=\bigl(\varphi^+\bigr)^k(\ta)$.
We have the following central proposition.
\begin{proposition}\label{prp_rotation}
Let $m$ be the non-crossing matching for $\autoseqplus$. Then
\begin{equation}\label{eqn_periodic}
\rho\bigl(\autoseqplus\bigr)=(\ta\tb\tc)^\omega.
\end{equation}
\end{proposition}
\begin{proof}
Let us first note that the limit itself can be obtained in a simpler way.
For any closed word $C$ over $K$,
\begin{equation*}
\mbox{\inparaenum[(1)]\item apply $\gamma$,\hspace{0.8em}\item remove all occurrences of $\tb$,\hspace{0.8em} \item reinsert $\tb$ before each $\tc$.}
\end{equation*}
The resulting word equals $\rho(C)$.
This statement simply follows from the facts that
{\inparaenum[(i)]
\item both procedures do not change the order in which the underlying letters $\ta$ and $\tc$ appear, that \item each occurrence of $\tc$ in both results is preceded by $\tb$, and that \item in both results, $\tb$ does not appear at other places.
}
We therefore see that Proposition~\ref{prp_rotation} is equivalent to the following.
Let $\mathbf C$ be the sequence obtained from $\autoseqplus$ by deleting all decorations, and all occurrences of $\tb$ and $\tbb$.
Then $\mathbf C=(\ta\tc)^\omega$.
In other words, we only have to show that $\ta$ and $\tc$ occur alternatingly in $\autoseq$, with the empty word or one occurrence of $\tb$ in between.
We prove a stronger statement concerning the sequence $\barautoseq$, which will complete the proof of Proposition~\ref{prp_rotation}.
\begin{lemma}
There are sequences $(\varepsilon_k)_{k\geq 0}$ and $(\varepsilon'_k)_{k\geq 0}$ in $\{0,1\}$ such that
\[\barautoseq=\ta
\bigl(\tb\tc(\ta\tc)^{\varepsilon_0}\tbb\ta(\tc\ta)^{\varepsilon'_0}\bigr)
\bigl(\tb\tc(\ta\tc)^{\varepsilon_1}\tbb\ta(\tc\ta)^{\varepsilon'_1}\bigr)
\cdots.
\]
\end{lemma}
In order to prove this, we apply the second iteration $\overline\varphi^2$ of Berstel's morphism on one of the expressions in brackets. We use the abbreviation
\[b(\varepsilon,\varepsilon')=\tb\tc(\ta\tc)^\varepsilon\tbb\ta(\tc\ta)^{\varepsilon'}.\]
Direct computation yields
\begin{multline*}
\begin{aligned}
\overline\varphi^2(b(0,0))
&=\tc\tbb\ta\tb\tc\tbb\ta\tc\ta\tb\tc\tbb\ta\tb\tc\ta
=\tc\tbb\ta\,b(0,1)b(0,0)\tb\tc\ta,\\
\overline\varphi^2(b(0,1))
&=\tc\tbb\ta\tb\tc\tbb\ta\tc\ta\tb\tc\tbb\ta\tb\tc\ta\tc\tbb\ta\tc\ta\tb\tc\ta
=\tc\tbb\ta\,b(0,1)b(0,0)b(1,1)\tb\tc\ta,\\
\overline\varphi^2(b(1,0))
&=\tc\tbb\ta \tb\tc\tbb\ta\tc\ta\tb\tc\ta\tc\tbb\ta\tc\ta\tb\tc\tbb\ta\tb\tc\ta
=\tc\tbb\ta\,b(0,1)b(1,1)b(0,0)\tb\tc\ta,\\
\overline\varphi^2(b(1,1))
&=\tc\tbb\ta \tb\tc\tbb\ta\tc\ta\tb\tc\ta\tc\tbb\ta\tc\ta\tb\tc\tbb\ta\tb\tc\ta\tc\tbb\ta\tc\ta\tb\tc\ta
\end{aligned}\\
=\tc\tbb\ta\,b(0,1)b(1,1)b(0,0)b(1,1)\tb\tc\ta.
\end{multline*}
Arbitrary concatenations of these expressions are again of the form $\tc\tbb\ta R\,\tb\tc\ta$, where $R$ is a concatenation of words $b(\varepsilon,\varepsilon')$.
Assuming that $w$ is of the form $\ta\prod_{j<r}b(\varepsilon_j,\varepsilon'_j)\tb\tc\ta$, we obtain
$\overline\varphi^2(w)=\ta\,b(1,0)R\,\tb\tc\ta$.
Since the words $\bigl(\overline\varphi^2\bigr)^k(\ta)$ approach a fixed point,
and
\[\overline\varphi^4(\ta)=\ta\,b(1,0)b(0,1)\tb\tc\ta,\]
it follows by induction that $\barautoseq$ is indeed of the form stated in the lemma, and we have in particular proved Proposition~\ref{prp_rotation}.
\end{proof}
From this algorithm, we can clearly see that a given letter $\ta$ is shifted, one place at a time, for each link that is passing over this letter. The direction in which $\ta$ is shifted depends on whether $\Lcright$ or $\Lcleft$ appears in the link we are dealing with.
We will use considerations of this kind in the following section,
²together with Proposition~\ref{prp_rotation}, in order to determine the \emph{discrepancy} of $\tO\tL$-occurrences in $\bt$.
\subsection{The discrepancy of $\tO\tL$-blocks}
For an integer $j\geq 0$
let us define the \emph{degree} of $j$ as follows. Let $m$ be the non-crossing matching for $\autoseqplus$ and set
\begin{equation}\label{eqn_def_degree}
\begin{aligned}
\deg^+(j)&=\#\bigl\{
(k,\ell)\in m: k<j<\ell\mbox{\textsf{ and }}\autoseqplus_k=\Lcright\bigr\},\\
\deg^-(j)&=\#\bigl\{
(k,\ell)\in m: k<j<\ell\mbox{\textsf{ and }}\autoseqplus_\ell=\Lcleft\bigr\},\\
\deg(j)&=\deg^+(j)-\deg^-(j).
\end{aligned}
\end{equation}
We will also talk about the degree of a  letter in $\autoseqplus$, where the position in question will always be clear from the context.

We display the first $192$ letters of $\autoseqplus$, obtained by applying the third iteration of $\varphi^+$ on the word $\autoseqplus_0\autoseqplus_1\autoseqplus_2=\ta\Lbright\!\Lcleft$,
and we connect associated connectors by actual lines for better readability.
On position $10=(\tZ\tZ)_4$ in~$\autoseqplus$, we have a letter $\ta$ of degree $-1$, and on position $170=(\tZ\tZ\tZ\tZ)_4$, a letter $\ta$ of degree $-2$.
These positions are marked with an arrow.
\begin{equation}\label{eqn_192_letters}
\begin{aligned}
&\begin{tikzpicture}[xscale=0.22,yscale=1,anchor=base,inner sep=1pt]
\node (a0)	at (0,0)	{$\ta$};
\node (a1)	at (1,0)	{$\tb$};
\node (a2)	at (2,0)	{$\tc$};
\node (a3)	at (3,0)	{$\ta$};
\node (a4)	at (4,0)	{$\tc$};
\node (a5)	at (5,0)	{$\tbb$};
\node (a6)	at (6,0)	{$\ta$};
\node (a7)	at (7,0)	{$\tb$};
\node (a8)	at (8,0)	{$\tc$};
\node (a9)	at (9,0)	{$\tbb$};
\node (a10)	at (10,0)	{$\ta$};
\node (a11)	at (11,0)	{$\tc$};
\node (a12)	at (12,0)	{$\ta$};
\node (a13)	at (13,0)	{$\tb$};
\node (a14)	at (14,0)	{$\tc$};
\node (a15)	at (15,0)	{$\ta$};
\node (a16)	at (16,0)	{$\tc$};
\node (a17)	at (17,0)	{$\tbb$};
\node (a18)	at (18,0)	{$\ta$};
\node (a19)	at (19,0)	{$\tc$};
\node (a20)	at (20,0)	{$\ta$};
\node (a21)	at (21,0)	{$\tb$};
\node (a22)	at (22,0)	{$\tc$};
\node (a23)	at (23,0)	{$\tbb$};
\node (a24)	at (24,0)	{$\ta$};
\node (a25)	at (25,0)	{$\tb$};
\node (a26)	at (26,0)	{$\tc$};
\node (a27)	at (27,0)	{$\ta$};
\node (a28)	at (28,0)	{$\tc$};
\node (a29)	at (29,0)	{$\tbb$};
\node (a30)	at (30,0)	{$\ta$};
\node (a31)	at (31,0)	{$\tb$};
\node (a32)	at (32,0)	{$\tc$};
\node (a33)	at (33,0)	{$\tbb$};
\node (a34)	at (34,0)	{$\ta$};
\node (a35)	at (35,0)	{$\tc$};
\node (a36)	at (36,0)	{$\ta$};
\node (a37)	at (37,0)	{$\tb$};
\node (a38)	at (38,0)	{$\tc$};
\node (a39)	at (39,0)	{$\tbb$};
\node (a40)	at (40,0)	{$\ta$};
\node (a41)	at (41,0)	{$\tb$};
\node (a42)	at (42,0)	{$\tc$};
\node (a43)	at (43,0)	{$\ta$};
\node (a44)	at (44,0)	{$\tc$};
\node (a45)	at (45,0)	{$\tbb$};
\node (a46)	at (46,0)	{$\ta$};
\node (a47)	at (47,0)	{$\tc$};
\draw (a1.south)  -- ++(0,-0.1) --  ++(1,0) --  (a2);
\draw (a4.south)  -- ++(0,-0.1) --  ++(1,0) --  (a5);
\draw (a7.south)  -- ++(0,-0.1) --  ++(1,0) --  (a8);
\draw (a9.south)  -- ++(0,-0.1) --  ++(2,0) --  (a11);
\draw[>=latex,->] (a10.north)++(0,0.3)  -- ++(0,-0.3);
\draw (a13.south)  -- ++(0,-0.1) --  ++(1,0) --  (a14);
\draw (a16.south)  -- ++(0,-0.1) --  ++(1,0) --  (a17);
\draw (a19.south)  -- ++(0,-0.1) --  ++(2,0) --  (a21);
\draw (a22.south)  -- ++(0,-0.1) --  ++(1,0) --  (a23);
\draw (a25.south)  -- ++(0,-0.1) --  ++(1,0) --  (a26);
\draw (a28.south)  -- ++(0,-0.1) --  ++(1,0) --  (a29);
\draw (a31.south)  -- ++(0,-0.1) --  ++(1,0) --  (a32);
\draw (a33.south)  -- ++(0,-0.1) --  ++(2,0) --  (a35);
\draw (a37.south)  -- ++(0,-0.1) --  ++(1,0) --  (a38);
\draw (a39.south)  -- ++(0,-0.2) --  ++(5,0) --  (a44);
\draw (a41.south)  -- ++(0,-0.1) --  ++(1,0) --  (a42);
\draw (a45.south)  -- ++(0,-0.1) --  ++(2,0) --  (a47);
\end{tikzpicture}
\\&
\begin{tikzpicture}[xscale=0.22,yscale=1,anchor=base, inner sep=1pt]
\node (a0)	at (0,0)	{$\ta$};
\node (a1)	at (1,0)	{$\tb$};
\node (a2)	at (2,0)	{$\tc$};
\node (a3)	at (3,0)	{$\ta$};
\node (a4)	at (4,0)	{$\tc$};
\node (a5)	at (5,0)	{$\tbb$};
\node (a6)	at (6,0)	{$\ta$};
\node (a7)	at (7,0)	{$\tb$};
\node (a8)	at (8,0)	{$\tc$};
\node (a9)	at (9,0)	{$\tbb$};
\node (a10)	at (10,0)	{$\ta$};
\node (a11)	at (11,0)	{$\tc$};
\node (a12)	at (12,0)	{$\ta$};
\node (a13)	at (13,0)	{$\tb$};
\node (a14)	at (14,0)	{$\tc$};
\node (a15)	at (15,0)	{$\ta$};
\node (a16)	at (16,0)	{$\tc$};
\node (a17)	at (17,0)	{$\tbb$};
\node (a18)	at (18,0)	{$\ta$};
\node (a19)	at (19,0)	{$\tc$};
\node (a20)	at (20,0)	{$\ta$};
\node (a21)	at (21,0)	{$\tb$};
\node (a22)	at (22,0)	{$\tc$};
\node (a23)	at (23,0)	{$\tbb$};
\node (a24)	at (24,0)	{$\ta$};
\node (a25)	at (25,0)	{$\tb$};
\node (a26)	at (26,0)	{$\tc$};
\node (a27)	at (27,0)	{$\ta$};
\node (a28)	at (28,0)	{$\tc$};
\node (a29)	at (29,0)	{$\tbb$};
\node (a30)	at (30,0)	{$\ta$};
\node (a31)	at (31,0)	{$\tc$};
\node (a32)	at (32,0)	{$\ta$};
\node (a33)	at (33,0)	{$\tb$};
\node (a34)	at (34,0)	{$\tc$};
\node (a35)	at (35,0)	{$\ta$};
\node (a36)	at (36,0)	{$\tc$};
\node (a37)	at (37,0)	{$\tbb$};
\node (a38)	at (38,0)	{$\ta$};
\node (a39)	at (39,0)	{$\tb$};
\node (a40)	at (40,0)	{$\tc$};
\node (a41)	at (41,0)	{$\tbb$};
\node (a42)	at (42,0)	{$\ta$};
\node (a43)	at (43,0)	{$\tc$};
\node (a44)	at (44,0)	{$\ta$};
\node (a45)	at (45,0)	{$\tb$};
\node (a46)	at (46,0)	{$\tc$};
\node (a47)	at (47,0)	{$\tbb$};
\draw (a1.south)  -- ++(0,-0.1) --  ++(1,0) --  (a2);
\draw (a4.south)  -- ++(0,-0.1) --  ++(1,0) --  (a5);
\draw (a7.south)  -- ++(0,-0.1) --  ++(1,0) --  (a8);
\draw (a9.south)  -- ++(0,-0.1) --  ++(2,0) --  (a11);
\draw (a13.south)  -- ++(0,-0.1) --  ++(1,0) --  (a14);
\draw (a16.south)  -- ++(0,-0.1) --  ++(1,0) --  (a17);
\draw (a19.south)  -- ++(0,-0.1) --  ++(2,0) --  (a21);
\draw (a22.south)  -- ++(0,-0.1) --  ++(1,0) --  (a23);
\draw (a25.south)  -- ++(0,-0.1) --  ++(1,0) --  (a26);
\draw (a28.south)  -- ++(0,-0.1) --  ++(1,0) --  (a29);
\draw (a31.south)  -- ++(0,-0.2) --  ++(8,0) --  (a39);
\draw (a33.south)  -- ++(0,-0.1) --  ++(1,0) --  (a34);
\draw (a36.south)  -- ++(0,-0.1) --  ++(1,0) --  (a37);
\draw (a40.south)  -- ++(0,-0.1) --  ++(1,0) --  (a41);
\draw (a43.south)  -- ++(0,-0.1) --  ++(2,0) --  (a45);
\draw (a46.south)  -- ++(0,-0.1) --  ++(1,0) --  (a47);
\end{tikzpicture}
\\
&\begin{tikzpicture}[xscale=0.22,yscale=1,anchor=base, inner sep=1pt]
\node (a0)	at (0,0)	{$\ta$};
\node (a1)	at (1,0)	{$\tb$};
\node (a2)	at (2,0)	{$\tc$};
\node (a3)	at (3,0)	{$\ta$};
\node (a4)	at (4,0)	{$\tc$};
\node (a5)	at (5,0)	{$\tbb$};
\node (a6)	at (6,0)	{$\ta$};
\node (a7)	at (7,0)	{$\tb$};
\node (a8)	at (8,0)	{$\tc$};
\node (a9)	at (9,0)	{$\tbb$};
\node (a10)	at (10,0)	{$\ta$};
\node (a11)	at (11,0)	{$\tc$};
\node (a12)	at (12,0)	{$\ta$};
\node (a13)	at (13,0)	{$\tb$};
\node (a14)	at (14,0)	{$\tc$};
\node (a15)	at (15,0)	{$\ta$};
\node (a16)	at (16,0)	{$\tc$};
\node (a17)	at (17,0)	{$\tbb$};
\node (a18)	at (18,0)	{$\ta$};
\node (a19)	at (19,0)	{$\tc$};
\node (a20)	at (20,0)	{$\ta$};
\node (a21)	at (21,0)	{$\tb$};
\node (a22)	at (22,0)	{$\tc$};
\node (a23)	at (23,0)	{$\tbb$};
\node (a24)	at (24,0)	{$\ta$};
\node (a25)	at (25,0)	{$\tb$};
\node (a26)	at (26,0)	{$\tc$};
\node (a27)	at (27,0)	{$\ta$};
\node (a28)	at (28,0)	{$\tc$};
\node (a29)	at (29,0)	{$\tbb$};
\node (a30)	at (30,0)	{$\ta$};
\node (a31)	at (31,0)	{$\tb$};
\node (a32)	at (32,0)	{$\tc$};
\node (a33)	at (33,0)	{$\tbb$};
\node (a34)	at (34,0)	{$\ta$};
\node (a35)	at (35,0)	{$\tc$};
\node (a36)	at (36,0)	{$\ta$};
\node (a37)	at (37,0)	{$\tb$};
\node (a38)	at (38,0)	{$\tc$};
\node (a39)	at (39,0)	{$\tbb$};
\node (a40)	at (40,0)	{$\ta$};
\node (a41)	at (41,0)	{$\tb$};
\node (a42)	at (42,0)	{$\tc$};
\node (a43)	at (43,0)	{$\ta$};
\node (a44)	at (44,0)	{$\tc$};
\node (a45)	at (45,0)	{$\tbb$};
\node (a46)	at (46,0)	{$\ta$};
\node (a47)	at (47,0)	{$\tc$};
\draw (a1.south)  -- ++(0,-0.1) --  ++(1,0) --  (a2);
\draw (a4.south)  -- ++(0,-0.1) --  ++(1,0) --  (a5);
\draw (a7.south)  -- ++(0,-0.1) --  ++(1,0) --  (a8);
\draw (a9.south)  -- ++(0,-0.1) --  ++(2,0) --  (a11);
\draw (a13.south)  -- ++(0,-0.1) --  ++(1,0) --  (a14);
\draw (a16.south)  -- ++(0,-0.1) --  ++(1,0) --  (a17);
\draw (a19.south)  -- ++(0,-0.1) --  ++(2,0) --  (a21);
\draw (a22.south)  -- ++(0,-0.1) --  ++(1,0) --  (a23);
\draw (a25.south)  -- ++(0,-0.1) --  ++(1,0) --  (a26);
\draw (a28.south)  -- ++(0,-0.1) --  ++(1,0) --  (a29);
\draw (a31.south)  -- ++(0,-0.1) --  ++(1,0) --  (a32);
\draw (a33.south)  -- ++(0,-0.1) --  ++(2,0) --  (a35);
\draw (a37.south)  -- ++(0,-0.1) --  ++(1,0) --  (a38);
\draw (a39.south)  -- ++(0,-0.2) --  ++(5,0) --  (a44);
\draw (a41.south)  -- ++(0,-0.1) --  ++(1,0) --  (a42);
\draw (a45.south)  -- ++(0,-0.1) --  ++(2,0) --  (a47);
\end{tikzpicture}
\\
&\begin{tikzpicture}[xscale=0.22,yscale=1, anchor=base,inner sep=1pt]
\node (a0)	at (0,0)	{$\ta$};
\node (a1)	at (1,0)	{$\tb$};
\node (a2)	at (2,0)	{$\tc$};
\node (a3)	at (3,0)	{$\ta$};
\node (a4)	at (4,0)	{$\tc$};
\node (a5)	at (5,0)	{$\tbb$};
\node (a6)	at (6,0)	{$\ta$};
\node (a7)	at (7,0)	{$\tb$};
\node (a8)	at (8,0)	{$\tc$};
\node (a9)	at (9,0)	{$\tbb$};
\node (a10)	at (10,0)	{$\ta$};
\node (a11)	at (11,0)	{$\tc$};
\node (a12)	at (12,0)	{$\ta$};
\node (a13)	at (13,0)	{$\tb$};
\node (a14)	at (14,0)	{$\tc$};
\node (a15)	at (15,0)	{$\tbb$};
\node (a16)	at (16,0)	{$\ta$};
\node (a17)	at (17,0)	{$\tb$};
\node (a18)	at (18,0)	{$\tc$};
\node (a19)	at (19,0)	{$\ta$};
\node (a20)	at (20,0)	{$\tc$};
\node (a21)	at (21,0)	{$\tbb$};
\node (a22)	at (22,0)	{$\ta$};
\node (a23)	at (23,0)	{$\tb$};
\node (a24)	at (24,0)	{$\tc$};
\node (a25)	at (25,0)	{$\tbb$};
\node (a26)	at (26,0)	{$\ta$};
\draw[>=latex,->] (a26.north)++(0,0.3)  -- ++(0,-0.3);
\node (a27)	at (27,0)	{$\tc$};
\node (a28)	at (28,0)	{$\ta$};
\node (a29)	at (29,0)	{$\tb$};
\node (a30)	at (30,0)	{$\tc$};
\node (a31)	at (31,0)	{$\ta$};
\node (a32)	at (32,0)	{$\tc$};
\node (a33)	at (33,0)	{$\tbb$};
\node (a34)	at (34,0)	{$\ta$};
\node (a35)	at (35,0)	{$\tc$};
\node (a36)	at (36,0)	{$\ta$};
\node (a37)	at (37,0)	{$\tb$};
\node (a38)	at (38,0)	{$\tc$};
\node (a39)	at (39,0)	{$\tbb$};
\node (a40)	at (40,0)	{$\ta$};
\node (a41)	at (41,0)	{$\tb$};
\node (a42)	at (42,0)	{$\tc$};
\node (a43)	at (43,0)	{$\ta$};
\node (a44)	at (44,0)	{$\tc$};
\node (a45)	at (45,0)	{$\tbb$};
\node (a46)	at (46,0)	{$\ta$};
\node (a47)	at (47,0)	{$\tc$};
\draw (a1.south)  -- ++(0,-0.1) --  ++(1,0) --  (a2);
\draw (a4.south)  -- ++(0,-0.1) --  ++(1,0) --  (a5);
\draw (a7.south)  -- ++(0,-0.1) --  ++(1,0) --  (a8);
\draw (a9.south)  -- ++(0,-0.1) --  ++(2,0) --  (a11);
\draw (a13.south)  -- ++(0,-0.1) -- ++(1,0) --  (a14);
\draw (a15.south)  -- ++(0,-0.2) -- ++(17,0) --  (a32);
\draw (a17.south)  -- ++(0,-0.1) -- ++(1,0) --  (a18);
\draw (a20.south)  -- ++(0,-0.1) -- ++(1,0) --  (a21);
\draw (a23.south)  -- ++(0,-0.1) -- ++(1,0) --  (a24);
\draw (a25.south)  -- ++(0,-0.1) -- ++(2,0) --  (a27);
\draw (a29.south)  -- ++(0,-0.1) -- ++(1,0) --  (a30);
\draw (a33.south)  -- ++(0,-0.1) --  ++(2,0) --  (a35);
\draw (a37.south)  -- ++(0,-0.1) --  ++(1,0) --  (a38);
\draw (a39.south)  -- ++(0,-0.2) --  ++(5,0) --  (a44);
\draw (a41.south)  -- ++(0,-0.1) --  ++(1,0) --  (a42);
\draw (a45.south)  -- ++(0,-0.1) --  ++(2,0) --  (a47);
\end{tikzpicture}
\end{aligned}
\end{equation}
From this initial segment we see that the sequence $(\deg(j))_{j\geq}$ starts with the $48$ integers
\[\begin{aligned}
&0,0,0,0,0,0,0,0,0,0,-1,0,0,0,0,0,0,0,0,0,1,0,0,0,0,0,0,0,0,0,0,0,0,0,
\\&-1,0,0,0,0,0,-1,-1,-1,-1,0,0,-1,0,\ldots,
\end{aligned}
\]
corresponding to the first line of~\eqref{eqn_192_letters}.
Applying the substitution $(\varphi^+)^2$ on $\Lbbright\hspace{-0.5pt}\ta\hspace{-0.5pt}\Lcleft$ appearing on positions $169$--$171$, we obtain the following $48$ letters.
The marked letter $\ta$ has degree $-3$, and it corresponds to the position $(\tZ\tZ\tZ\tZ\tZ\tZ)_4$.
\begin{equation*}
\begin{tikzpicture}[xscale=0.22,yscale=1, anchor=base,inner sep=1pt]
\node (a0)	at (0,0)	{$\ta$};
\node (a1)	at (1,0)	{$\tb$};
\node (a2)	at (2,0)	{$\tc$};
\node (a3)	at (3,0)	{$\ta$};
\node (a4)	at (4,0)	{$\tc$};
\node (a5)	at (5,0)	{$\tbb$};
\node (a6)	at (6,0)	{$\ta$};
\node (a7)	at (7,0)	{$\tb$};
\node (a8)	at (8,0)	{$\tc$};
\node (a9)	at (9,0)	{$\tbb$};
\node (a10)	at (10,0)	{$\ta$};
\node (a11)	at (11,0)	{$\tc$};
\node (a12)	at (12,0)	{$\ta$};
\node (a13)	at (13,0)	{$\tb$};
\node (a14)	at (14,0)	{$\tc$};
\node (a15)	at (15,0)	{$\tbb$};
\node (a16)	at (16,0)	{$\ta$};
\node (a17)	at (17,0)	{$\tb$};
\node (a18)	at (18,0)	{$\tc$};
\node (a19)	at (19,0)	{$\ta$};
\node (a20)	at (20,0)	{$\tc$};
\node (a21)	at (21,0)	{$\tbb$};
\node (a22)	at (22,0)	{$\ta$};
\node (a23)	at (23,0)	{$\tb$};
\node (a24)	at (24,0)	{$\tc$};
\node (a25)	at (25,0)	{$\tbb$};
\draw[>=latex,->] (a26.north)++(0,0.3)  -- ++(0,-0.3);
\node (a26)	at (26,0)	{$\ta$};
\node (a27)	at (27,0)	{$\tc$};
\node (a28)	at (28,0)	{$\ta$};
\node (a29)	at (29,0)	{$\tb$};
\node (a30)	at (30,0)	{$\tc$};
\node (a31)	at (31,0)	{$\ta$};
\node (a32)	at (32,0)	{$\tc$};
\node (a33)	at (33,0)	{$\tbb$};
\node (a34)	at (34,0)	{$\ta$};
\node (a35)	at (35,0)	{$\tc$};
\node (a36)	at (36,0)	{$\ta$};
\node (a37)	at (37,0)	{$\tb$};
\node (a38)	at (38,0)	{$\tc$};
\node (a39)	at (39,0)	{$\tbb$};
\node (a40)	at (40,0)	{$\ta$};
\node (a41)	at (41,0)	{$\tb$};
\node (a42)	at (42,0)	{$\tc$};
\node (a43)	at (43,0)	{$\ta$};
\node (a44)	at (44,0)	{$\tc$};
\node (a45)	at (45,0)	{$\tbb$};
\node (a46)	at (46,0)	{$\ta$};
\node (a47)	at (47,0)	{$\tc$};
\draw (a0.south)++(-0.5,-0.3) -- ++(48,0);
\draw [dash pattern=on1.5off1.5] (a0.south)++(-2.5,-0.3) -- ++(2,0);
\draw [dash pattern=on1.5off1.5] (a0.south)++(47.5,-0.3) -- ++(2.3,0);
\draw (a1.south)  -- ++(0,-0.1) --  ++(1,0) --  (a2);
\draw (a4.south)  -- ++(0,-0.1) --  ++(1,0) --  (a5);
\draw (a7.south)  -- ++(0,-0.1) --  ++(1,0) --  (a8);
\draw (a9.south)  -- ++(0,-0.1) --  ++(2,0) --  (a11);
\draw (a13.south)  -- ++(0,-0.1) --  ++(1,0) --  (a14);
\draw (a15.south)  -- ++(0,-0.2) --  ++(17,0) -- (a32);
\draw (a17.south)  -- ++(0,-0.1) --  ++(1,0) --  (a18);
\draw (a20.south)  -- ++(0,-0.1) --  ++(1,0) --  (a21);
\draw (a23.south)  -- ++(0,-0.1) --  ++(1,0) --  (a24);
\draw (a25.south)  -- ++(0,-0.1) --  ++(2,0) --  (a27);
\draw (a29.south)  -- ++(0,-0.1) --  ++(1,0) --  (a30);
\draw (a33.south)  -- ++(0,-0.1) --  ++(2,0) --  (a35);
\draw (a37.south)  -- ++(0,-0.1) --  ++(1,0) --  (a38);
\draw (a39.south)  -- ++(0,-0.2) --  ++(5,0) --  (a44);
\draw (a41.south)  -- ++(0,-0.1) --  ++(1,0) --  (a42);
\draw (a45.south)  -- ++(0,-0.1) --  ++(2,0) --  (a47);
\end{tikzpicture}
\end{equation*}
In general, on position $(\tZ^{2k})_4$, a letter $\ta$ of degree $-k$ appears.
This can be seen by considering the images of $\ta$ and $\Lcleft$ under $\varphi^+$.

By Proposition~\ref{prp_rotation}, $\deg(j)$ has the following meaning in the case that $\autoseqplus_j=\ta$.
A number of $\deg^+(j)$ letters $\Lbleft$ is transferred from the right of the letter $\ta$ to the left of it; note that the letter $\ta$ is shifted to the right $\deg^+(j)$ places.
Analogously, $\deg^-(j)$ letters $\Lbbright$ are transferred from the left of $\ta$ to the right, and the letter $\ta$ is shifted to the left $\deg^-(j)$ places.
In total, the letter $\ta$ (among other letters) is shifted by $\deg(j)$ places, and $\tb$s or $\tbb$s are moved to account for the generated trailing space.
The proposition states that the letters to the left of $\ta$'s new position $j+\deg^+(j)$ are balanced --- after removing decorations and replacing $\tbb$ by $\tb$, the letters $\ta$, $\tb$, and $\tc$ occur the same number of times.
If $\autoseqplus_j\in\{\Lcleft,\Lcright\}$, similar considerations hold.
The case of letters of degree $\tb$ is different, since a single rotation may shift such a letter to a remote place.

\newcommand{\radius}{3}
\newcommand{\rot}{15}
\begin{figure}
\begin{center}
\begin{tikzpicture}[	description/.style={fill=white,inner sep=1pt},
						cross line/.style={preaction={draw=white,-,line width=3pt}},
						scale=0.9]

	\node (c)  at (0,0)	{$\ta$};
	\node[circle,minimum size=6mm,draw] (ccopy)  at (0,0) {};
	\node (n0)	at ({0*360/6}:\radius)	{$\Lbright$};
	\node[circle,minimum size=6mm,draw] (ncopy0)  at ({0*360/6}:\radius) {};
	\node (n1)  at ({1*360/6}:\radius)	{$\Lbbright$};
	\node[circle,minimum size=6mm,draw] (ncopy1)  at ({1*360/6}:\radius) {};
	\node (n2)	at ({2*360/6}:\radius)	{$\Lcleft$};
	\node[circle,minimum size=6mm,draw] (ncopy2)  at ({2*360/6}:\radius) {};
	\node (n3)  at ({3*360/6}:\radius)	{$\Lbleft$};
	\node[circle,minimum size=6mm,draw] (ncopy3)  at ({3*360/6}:\radius) {};
	\node (n4)	at ({4*360/6}:\radius)	{$\Lbbleft$};
	\node[circle,minimum size=6mm,draw] (ncopy4)  at ({4*360/6}:\radius) {};
	\node (n5)	at ({5*360/6}:\radius)	{$\Lcright$};
	\node[circle,minimum size=6mm,draw] (ncopy5)  at ({5*360/6}:\radius) {};

 \path[>=latex,->,font=\scriptsize] (ncopy0) edge [out=0*360/6+\rot,in=0*360/6-\rot,looseness=10] node[auto] {$\tiny{\tM\vert0}$} (ncopy0)
(ncopy1) edge [out=1*360/6+\rot,in=1*360/6-\rot,looseness=10] node[auto,above=1pt] {$\tM\vert0$} (ncopy1)
(ncopy2) edge [out=2*360/6+\rot,in=2*360/6-\rot,looseness=10] node[auto,above=1pt] {$\tO\vert0,\tM\vert0$} (ncopy2)
(ncopy3) edge [out=3*360/6+\rot,in=3*360/6-\rot,looseness=10] node[auto] {$\tM\vert0$} (ncopy3)
(ncopy4) edge [out=4*360/6+\rot,in=4*360/6-\rot,looseness=10] node[auto,below=1pt] {$\tM\vert0$} (ncopy4)
(ncopy5) edge [out=5*360/6+\rot,in=5*360/6-\rot,looseness=10] node[auto,below=1pt] {$\tO\vert0,\tM\vert0$} (ncopy5)
(ccopy) edge [out=270+\rot,in=270-\rot,looseness=10] node[auto] {$\tiny{\tO\vert0,\tM\vert0}$} (ccopy);

	\node[rectangle, draw] (start)	at ({2.6*360/6}:{\radius*1})	{\scriptsize start};
 \path[>=latex,->,font=\scriptsize]
  (start) edge (ccopy)
(ncopy1) edge node[auto] {$\tL\vert0$} (ncopy0)
(ncopy1.{180-\rot}) edge node[auto, swap] {$\tZ\vert0$} (ncopy2.{\rot})
(ncopy2.{360-\rot}) edge node[auto, swap] {$\tL\vert0$} (ncopy1.{180+\rot})
(ncopy0) edge node[auto,near end] {$\tO\vert0$} (ncopy5)
(ncopy5.{180-\rot}) edge node[auto, swap] {$\tL\vert0$} (ncopy4.{\rot})
(ncopy4.{360-\rot}) edge node[auto, swap] {$\tZ\vert0$} (ncopy5.{180+\rot})
(ncopy4.{120-\rot}) edge node[auto, swap] {$\tL\vert0$} (ncopy3.{300+\rot})
(ncopy3.{300-\rot}) edge node[auto, swap,near end, below left=-3pt] {$\tL\vert1$} (ncopy4.{120+\rot})
(ccopy.{360-\rot}) edge node[auto, swap] {$\tL\vert0$} (ncopy0.{180+\rot})
(ncopy0.{180-\rot}) edge node[auto, swap] {$\tZ\vert0$} (ccopy.{\rot})
(ccopy.{120-\rot}) edge node[auto, swap] {$\tZ\vert0$} (ncopy2.{300+\rot})
(ncopy2.{300-\rot}) edge node[auto, swap, near start, below left=-3pt] {$\tZ\vert\!-\!1$} (ccopy.{120+\rot})
(ncopy1) edge node[auto] {$\tO\vert0$} (ccopy)
(ncopy3) edge node[auto, swap] {$\tZ\vert1$} (ccopy)
(ncopy4) edge node[auto, swap, near start, below right=-3pt] {$\tO\vert1$} (ccopy)
(ncopy5) edge node[auto, swap] {$\tZ\vert0$} (ccopy);

\draw[>=latex,->,font=\scriptsize] (ncopy3)  -- ++(0,{-\radius*4/3}) --node [auto,swap] (e1)	{$\tO\vert1$} ++({\radius*5/6},0) -- (ncopy5);
\draw[>=latex,->,font=\scriptsize,cross line] (ncopy0)  -- ++(0,{-\radius*4/3}) --node [auto] (e2)	{$\tL\vert0$} ++({-\radius*5/6},0) -- (ncopy4);

\end{tikzpicture}
\end{center}
\caption{A base-$4$ transducer that generates the degree sequence}
\label{fig_hexagon}
\end{figure}
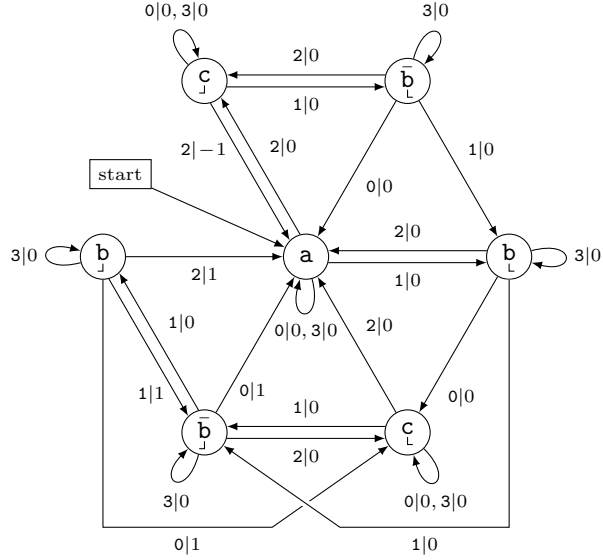


The transducer $\mathcal T_1$ displayed in Figure~\ref{fig_hexagon} allows us to compute the degree of an arbitrary position $j$: starting from the centre node, we traverse the graph, guided by the base-$4$ expansion $\delta_{\nu-1}\cdots\delta_0$ of $j$ (read from left to right).
Along the way, we sum up the numbers $k$ whenever a vertex $\delta_i\mid k$ is taken.
The sum over these numbers is the degree of $j$, multiplied by $3$.
The transducer $\mathcal T_1$ is derived directly from the decorated, $4$-uniform morphism $\varphi^+$ given in~\eqref{eqn_decorated_substitution}.
Note that a change of degree takes place whenever new letters are inserted, by virtue of the morphism $\varphi^+$, into the range of already existing links, which happens for $\Lbleft$ and $\Lbbleft$; or, if a new link together with a letter $\ta$ in its range is created, which happens for $\Lcleft$.

We will now apply Proposition~\ref{prp_rotation} to the discrepancy $D_N$ of occurrences of $\tO\tL$ in $\bt$.

\begin{proposition}\label{prp_discrepancy_from_degree}
Let $j\in\mathbb N$ and set $d=\deg(j)$.
Then
\begin{equation}\label{eqn_discrepancy_from_degree}
\begin{array}{rll@{\hspace{4em}}rll}
D_{4j}&=d/3,&\mbox{if }\autoseqplus_j=\ta;\\
D_{4j}&=d/3+1/3,&\mbox{if }\autoseqplus_j=\Lbbleft;\\
D_{4j}&=d/3,&\mbox{if }\autoseqplus_j=\Lbbright;\\[3mm]
D_{4j+2}&=d/3+1/3,&\mbox{if }\autoseqplus_j=\Lbleft;\\
D_{4j+2}&=d/3,&\mbox{if }\autoseqplus_j=\Lbright;\\
D_{4j+2}&=d/3-1/3,&\mbox{if }\autoseqplus_j=\Lcleft;\\
D_{4j+2}&=d/3,&\mbox{if }\autoseqplus_j=\Lcright.
\end{array}
\end{equation}
In each of these cases, the subscript of $D$ is the position in $\bt$ that corresponds to the $j$th letter in $\autoseq$ via~\eqref{eqn_TM_autoseq_concatenation}.

\end{proposition}
\begin{proof}
Choose $\varepsilon\in\{0,1,2\}$ and $n\in\mathbb N$ such that $j=3n+\varepsilon$.
Let us consider each of the seven cases corresponding to letters from $K$.

\textbf{First case.} Assume that $\autoseqplus_j=\ta$.
By Algorithm \texttt{RotateAlongLinks} and Proposition~\ref{prp_rotation},
a total of $d$ letters of type $\tb$ have to be shifted from the right of our $\ta$ in question to the left (if $d>0$), or the other way round (if $d<0$).
After this procedure, the numbers of letters of types $\ta$, $\tb$, and $\tc$ to the left are equal.
It follows that $\varepsilon\equiv-d\bmod 3$;
moreover, $m=n+(\varepsilon+d)/3$ is the number of letters $\ta$ (and also the number of letters of type $\tc$) strictly to the left of $j$.
The number of letters of type $\tb$ to the left of $j$ is $m'=n+(\varepsilon-2d)/3$.
Symbols of type $\ta$ contribute two blocks $\tO\tL$ and correspond to a factor of length six in $\bt$, by~\eqref{eqn_TM_abc_correspondence};
letters of type $\tb$ contribute one block and correspond to a factor of length four; letters of type $\tc$ contribute one block and correspond to a factor of length two.
It follows that below position
\begin{equation*}
N=(6+2)\left(n+\frac{\varepsilon+d}3\right)+4\left(n+\frac{\varepsilon-2d}3\right)
=12n+4\,\varepsilon=4j,
\end{equation*}
we find
\[(2+1)\left(n+\frac{\varepsilon+d}3\right)+\left(n+\frac{\varepsilon-2d}3\right)
=4n+4\,\varepsilon/3+d/3
\]
blocks $\tO\tL$.
This proves the case $\autoseqplus_j=a$.

\textbf{Second case.}
If $\autoseqplus_j=\Lbbleft$, we note that necessarily $\autoseqplus_{j+1}=\ta$,
by Lemma~\ref{lem_local}.
We apply the first case on position $j+1$, which has degree $d$.
Noting that a letter of type $\tb$ in $\autoseqplus$ corresponds to $\tO\tL\tL\tO$ in Thue--Morse,
we obtain $D_{4j}=D_{4j+4}+1/3=d/3+1/3$, where $4j$ resp. $4j+2$ correspond to the $j$th resp. $(j+1)$th position in $\autoseqplus$.

\textbf{Third case.}
Assume that $\autoseqplus_j=\Lbbright$.
In this case, the letter at $j+1$ is $\ta$ by Lemma~\ref{lem_local}
and $j+1$ has degree $d-1$.
It follows that
$D_{4j}=D_{4j+4}+1/3=d/3-1/3+1/3=d/3$.

\textbf{Fourth case.}
If $\autoseqplus_j=\Lbleft$, we note that necessarily $\autoseqplus_{j-1}=\ta$;
we apply the first case on position $j-1$, which has degree $d+1$.
since $\ta$ corresponds to $\tO\tL\tL\tO\tL\tO$ in Thue--Morse, we have
$D_{4j+2}=D_{4j-4}=d/3+1/3$, where $4j-4$ resp. $4j+2$ correspond to the $(j-1)$th resp. $j$th positions in $\autoseqplus$.

\textbf{Fifth case.} Assume that $\autoseqplus_j=\Lbright$. Then $\autoseqplus_{j-1}=\ta$, and $j-1$ has degree $d$.
Analogously to the fourth case, we obtain
$D_{4j+2}=D_{4j-4}=d/3$.

\textbf{Sixth case.}
If $\autoseqplus_j=\Lcleft$, this letter is connected to a letter of type $\tb$ to the left, which stays on the left of $\Lcleft$ after applying the rotations.
Therefore the number of letters of type $\tb$ to the left are changed by $d$, and the numbers of letters of types $\ta$ or $\tc$ to the left stay the same.
Similarly to the first case, it follows that $\varepsilon\equiv 2-d\bmod 3$.
The numbers of letters, of type $\ta$, $\tb$, and $\tc$ respectively, to the left of $j$, are therefore
$m=n+(\varepsilon+d+1)/3$, $m-d$, and $m-1$ respectively.
It follows that, below position
\begin{align*}
N&=6\left(n+\frac{\varepsilon+d+1}3\right)
+4\left(n+\frac{\varepsilon-2d+1}3\right)
+2\left(n+\frac{\varepsilon+d-2}3\right)
\\&=12n+4\,\varepsilon+2=4j+2
,
\end{align*}
there are
\[2\left(n+\frac{\varepsilon+d+1}3\right)
  +\left(n+\frac{\varepsilon-2d+1}3\right)
  +\left(n+\frac{\varepsilon+d-2}3\right)
\!=4n+\frac{4\,\varepsilon}3+\frac d3+\frac13
\]
blocks $\tO\tL$.

\textbf{Seventh case.}
If $\autoseqplus_j=\Lcright$, a letter of type $\tb$ is taking its place after one rotation.
In this case, we have $\varepsilon\equiv 1-d\bmod 3$;
the numbers of letters to the left of $j$, of types $\ta$, $\tb$, and $\tc$ respectively, are therefore
$m=n+(\varepsilon+d+2)/3$, $m-d-1$, and $m-1$ respectively.
Therefore, below position
\begin{align*}
N&=6\left(n+\frac{\varepsilon+d+2}3\right)
+4\left(n+\frac{\varepsilon-2d-1}3\right)
+2\left(n+\frac{\varepsilon+d-1}3\right)
\\&=12n+4\,\varepsilon+2=4j+2
,
\end{align*}
there are
\[2\left(n+\frac{\varepsilon+d+2}3\right)
  +\left(n+\frac{\varepsilon-2d-1}3\right)
  +\left(n+\frac{\varepsilon+d-1}3\right)
\!=4n+\frac{4\,\varepsilon}3+\frac d3+\frac23
\]
blocks $\tO\tL$, which proves the last case.
\end{proof}

Since $\deg(j)$ is easy to obtain, Proposition~\ref{prp_discrepancy_from_degree} gives us a simple method to compute the discrepancy $D_N$ for any given $N$.
\begin{proposition}\label{prp_regular}
Let $N\geq 0$ be an integer and $j=\lfloor N/4\rfloor$.
\begin{enumerate}
\item
If $\autoseqplus_j\in\{\ta,\Lbbleft,\Lbbright\}$,
choose $\delta=D_{4j}/3=\deg(j)/3+\varepsilon$, where $\varepsilon\in\{0,1/3\}$ is given by the first block of~\eqref{eqn_discrepancy_from_degree}.
Then
\begin{equation}\label{eqn_even_case}
\bigl(D_{4j},D_{4j+1},D_{4j+2},D_{4j+3}\bigr)=
\bigl(\delta,\delta+2/3,\delta+1/3,\delta\bigr).
\end{equation}
\item If $\autoseqplus_j\in\{\Lbleft,\Lbright,\Lcleft,\Lcright\}$, choose $\delta=D_{4j+2}/3=\deg(j)/3+\varepsilon$, where $\varepsilon\in\{-1/3,0,1/3\}$ is given by the second block of~\eqref{eqn_discrepancy_from_degree}.
Then
\begin{equation}\label{eqn_odd_case}
\bigl(D_{4j},D_{4j+1},D_{4j+2},D_{4j+3}\bigr)=
\bigl(\delta+2/3,\delta+1/3,\delta,\delta+2/3\bigr).
\end{equation}
\end{enumerate}
\end{proposition}

The scaled sequence of discrepancies (multiplied by $3$) therefore begins with the $48$ integers
\[
\begin{array}{llllll}
0,2,1,0,
&2,1,0,2,
&1,0,-1,1,
&0,2,1,0,
&2,1,0,2,
&1,3,2,1,
\\0,2,1,0,
&2,1,0,2,
&1,0,-1,1,
&0,2,1,0,
&-1,1,0,-1,
&1,0,-1,1.
\end{array}
\]

The partition into segments of length four is for better readability.
Each segment corresponds to one symbol in $\autoseqplus$.
\begin{proof}[Proof of Proposition~\ref{prp_regular}]
For the first sentence of each of the two cases, there is nothing to show, by Proposition~\ref{prp_discrepancy_from_degree}.
Let us begin with the first case.
By the proposition, the position $4j$ in the Thue--Morse sequence corresponds to a letter $\ta$ or $\tb$ in $\autoseq$ (on position $j$), and by~\eqref{eqn_TM_abc_correspondence} we have $(\bt_{4j},\bt_{4j+1},\bt_{4j+2},\bt_{4j+3})=(\tO\tL\tL\tO)$.
Therefore~\eqref{eqn_even_case} follows.
Concerning the second case, Proposition~\ref{prp_discrepancy_from_degree}
gives us an expression for $D_{4j+2}$ in terms of $\deg(j)$,
and the position $4j+2$ corresponds to the index $j$ in $\autoseq$.
By~\eqref{eqn_TM_abc_correspondence}, we have
$(\bt_{4j+2},\bt_{4j+3})=(\tO,\tL)$.
Therefore
\[\bigl(D_{4j+2},D_{4j+3}\bigr)=
\bigl(\delta,\delta+2/3\bigr).\]
In order to compute $D_{4j}$ and $D_{4j+1}$ in this case,
we note that $\Lbleft$ and $\Lbright$ are always preceded by $\ta$ (as we noted in the proof of Proposition~\ref{prp_discrepancy_from_degree}), and $\Lcleft$ and $\Lcright$ are always preceded by a letter of type $\ta$ or $\tb$, since $\autoseq$ is squarefree.
It follows that the letter at index $j-1$ is of type $\ta$ or $\tb$, and therefore
$(\bt_{4j},\bt_{4j+1})=(\tL\tO)$.
Consequently, we have
\[\bigl(D_{4j},D_{4j+1}\bigr)=
\bigl(\delta+2/3,\delta+1/3\bigr),
\]
and~\eqref{eqn_odd_case} follows.
\end{proof}

\subsection{Proof of Theorem~\ref{thm_discrepancy}}
We may now show that the sequence $(D_N)_{N\geq 0}$ of discrepancies is given by a base-$2$ transducer.
The transducer in Figure~\ref{fig_hexagon} may be described by eight $7\times7$-matrices $A^{(\ell)}, W^{(\ell)}$, for $0\leq \ell<4$, where rows and columns are indexed by the letters of $K$, in the order $(\ta\Lbbleft,\Lbbright,\Lbleft,\Lbright,\Lcleft,\Lcright)$.

The entry $A^{(\ell)}_{i,j}$ equals $1$ if there is an arrow with first component equal to $\ell$ from the $j$th node to the $i$th node in Figure~\ref{fig_hexagon}, and it is zero otherwise.
The matrices $A^{(\ell)}$ are permutation matrices.
The entry $W^{(\ell)}_{i,j}$ is the second component of the arrow from $j$ to $i$ with first component $\ell$, if there is one, and equal to zero otherwise.

The final modification given by~\eqref{eqn_even_case} and~\eqref{eqn_odd_case} 
is dealt with by four more matrices $Z^{(\ell)}$, where $0\leq \ell<4$.
The first three columns of these matrices are given by~\eqref{eqn_even_case}, as follows.
Define the quadruple $(q_0,q_1,q_2,q_3)=(0,2/3,1/3,0)$ (containing the shifts in~\eqref{eqn_even_case}), and the triple $(r_1,r_2,r_3)=(0,1/3,0)$ (taking care of the shifts present in the first block of~\eqref{eqn_discrepancy_from_degree}).
Let $1\leq j\leq 3$ (corresponding to the letter at which an arrow starts), and $0\leq \ell<4$ (a base-$4$ digit; the first component of the label of the arrow).
There is a unique $i\in\{1,\ldots,7\}$ such that $A^{(\ell)}_{i,j}=1$, and we set $Z^{(\ell)}_{i,j}=q_\ell+r_j$, and $Z^{(\ell)}_{i',j}=0$ for $i'\neq i$.
The remaining four columns are filled with the help of~\eqref{eqn_odd_case}, as follows.
Define $(\tilde q_0,\tilde q_1,\tilde q_2,\tilde q_3)=(2/3,1/3,0,2/3)$ and $(r_4,r_5,r_6,r_7)=(1/3,0,-1/3,0)$.
Let $4\leq j\leq 7$ and $0\leq \ell<4$.
There is a unique $i\in\{1,\ldots,7\}$ such that $A^{(\ell)}_{i,j}=1$, and we set $Z^{(\ell)}_{i,j}=\tilde q_\ell+r_j$, and $Z^{(\ell)}_{i',j}=0$ for $i'\neq i$.

In order to generate the discrepancy, we blow up the transducer by a factor
$28$, in order to keep track of the arrow that led to the current node (that is, we need to save the previously read digit $\ell'\in\{0,1,2,3\}$ and the node in $\mathcal T_1$ that was last visited).

In each step, the contribution of $Z^{(\ell')}$ is cancelled out, and the contributions of $A^{(\ell')}$ and $Z^{(\ell)}$ are added (where $\ell$ is the currently read digit). More precisely, let $(i,\ell',j)$, for $1\leq i,j\leq 7$ and $0\leq \ell'<4$, be the $196$ nodes of our new transducer $\mathcal T_2$.
There is an arrow from $(j,\ell',k)$ to $(i,\ell,j')$ if and only if
$j=j'$ and $A^{(\ell)}_{j,i}=1$ --- that is, if there is an arrow from $j$ to $i$ in $\mathcal T_1$ whose label has $\ell$ as its first component.
We may now define the weight of an arrow $(j,\ell',k)\rightarrow (i,\ell,j)$ as
\[Z^{(\ell)}_{i,j}-Z^{(\ell')}_{j,k}+W^{(\ell')}_{j,k}.\]
The initial node is $(1,0,1)$,
which corresponds to the fact that leading zeros do not make a difference.
Let us illustrate, by a short but representative example, the easy proof that the transducer $\mathcal T_2$ generates the discrepancy sequence.
We wish to compute the discrepancy $D_{41}=D_{(\tZ\tZ\tL)_4}$.
The corresponding path in $\mathcal T_2$ is given by
\[(1,0,1)\longrightarrow (5,2,1)\longrightarrow (1,2,5)\longrightarrow (4,1,1).\]
Note that the first and third components correspond to letters in $K$, that is, to nodes in $\mathcal T_1$, via
$1\rightleftharpoons\ta$, $4\rightleftharpoons \Lbright$, and $5\rightleftharpoons\Lcleft$.
The sum of the weights simplifies, due to a telescoping sum and $W^{(0)}_{1,1}=Z^{(0)}_{1,1}=0$, to
\[W^{(2)}_{5,1}+W^{(2)}_{1,5}+
Z^{(1)}_{4,1}.
\]
The first two summands sum up to $\deg((\tZ\tZ)_4)=-1/3$ by the construction of our transducer,
while the last summand consists of two parts: the shift in the first line of the first block of~\eqref{eqn_discrepancy_from_degree} (which is $0$), and the shift in the second component of~\eqref{eqn_even_case} (which is $2/3$).
Summing up, we obtain $D_{41}=1/3$.
It is clear that the proof of the general case is not more complicated this example.

Since the integers $2$ and $4$ are multiplicatively dependent, in symbols, $2^m=4^n$ for $(m,n)=(2,1)$, the sequence $D$ is also generated by a base-$2$ transducer.
In order to carry out this reduction to base two, the four arrows starting from a given node in our base-$4$ transducer have to be replaced by a complete binary tree of depth $2$, where two auxiliary nodes have to be inserted.
The proposition is proved, and thus the first part of Theorem~\ref{thm_discrepancy}.

The output sum of a base-$q$ transducer is clearly bounded by a constant times the length of the base-$q$ expansion we feed into the transducer.
This immediately yields $D_N\ll \log N$.

We easily see from Figure~\ref{fig_hexagon} that the integers
\[
(\tZ^{2k})_4=2\,\frac{16^k-1}{3}
\quad\mbox{and}\quad
((\tL\tL\tO)^k)_4
=20\,\frac{64^k-1}{63}
\]
have degrees $-k$ and $k$ respectively, for $k\geq 1$,
and that the letter $\ta$ is attained at these positions.
Therefore Proposition~\ref{prp_discrepancy_from_degree} implies
\begin{equation}\label{eqn_explicit_discrepancy}
D_{8\hspace{0.5pt}(16^k-1)/3}=-k/3\quad\mbox{and}\quad
D_{80\hspace{0.5pt}(64^k-1)/63}=k/3
\end{equation}
for $k\geq 1$, and clearly $D_0=0$.
In particular,
$\{D_N:N\geq 0\}=(1/3)\mathbb Z$,
which finishes the proof of Theorem~\ref{thm_discrepancy}.
\qed

By considering the path given by $n'=(\tZ^{2k-1})_4$ instead, we end up in the node $\Lcleft$, and the position $n'$ has degree $-k+1$.
Proposition~\ref{prp_regular} implies
$D_n=-k/3$, where $n=4n'+2=((\tL\tO)^{4k})_2$.
This was observed by Jeffrey Shallit (private communication, 2021), but such an unboundedness result does not seem to be stated in the literature.

\subsection*{Acknowledgements.}
The author thanks Jeff Shallit for sharing with him the research question treated in this paper, constant interest in his research, and quick and informative answers to e-mails.
The question was presented to the author during the workshop `Numeration and Substitution' at the ESI Vienna (Austria) in July 2019.
We express our thanks to the ESI for providing optimal working conditions at the workshop, including offices for participants and blackboards in unconventional places.
Finally, we thank Michel Rigo for fruitful discussions on the topic and Clemens M\"ullner for pointing out that the case of arbitrary factors can be reduced to the case $\mathtt 0\mathtt 1$.
%
%
%

\bibliographystyle{line}
\bibliography{H}
\end{document}